\documentclass[11pt,leqno]{amsart}
\usepackage{amssymb,amsmath,amsthm,amsfonts}
\usepackage {latexsym}
\usepackage{bbm}

\allowdisplaybreaks
\newcommand{\nc}{\newcommand}
\nc{\les}{\lesssim}
\nc{\nit}{\noindent}
\nc{\nn}{\nonumber}
\nc{\D}{\partial}
\nc{\diff}[2]{\frac{d #1}{d #2}}
\nc{\diffn}[3]{\frac{d^{#3} #1}{d {#2}^{#3}}}
\nc{\pdiff}[2]{\frac{\partial #1}{\partial #2}}
\nc{\pdiffn}[3]{\frac{\partial^{#3} #1}{\partial{#2}^{#3}}}
\nc{\abs}[1] {\lvert #1 \rvert}
\nc{\cAc}{{\cal A}_c}
\nc{\cE}{{\cal E}}
\nc{\cF}{{\cal F}}
\nc{\cP}{{\cal P}}
\nc{\cV}{{\cal V}}
\nc{\cQ}{{\cal Q}}
\nc{\cGin}{{\cal G}_{\rm in}}
\nc{\cGout}{{\cal G}_{\rm out}}
\nc{\cO}{{\cal O}}
\nc{\Lav}{{\cal L}_{\rm av}}
\nc{\cL}{{\cal L}}
\nc{\cB}{{\cal B}}
\nc{\cZ}{{\cal Z}}
\nc{\cR}{{\cal R}}
\nc{\cT}{{\cal T}}
\nc{\cY}{{\cal Y}}
\nc{\cX}{{\cal X}}
\nc{\cXT}{{{\cal X}(T)}}
\nc{\cBT}{{{\cal B}(T)}}
\nc{\vD}{{\vec \mathcal{D}}}
\nc{\efield}{\mathcal{E}}
\nc{\vE}{{\vec \efield}}
\nc{\vB}{{\vec \mathcal{B}}}
\nc{\vH}{{\vec \mathcal{H}}}
\nc{\ty}{{\tilde y}}
\nc{\tu}{{\tilde u}}
\nc{\tV}{{\tilde V}}
\nc{\Pc}{{\bf P_c}}
\nc{\bx}{{\bf x}}
\nc{\bX}{{\bf X}}
\nc{\bXYZ}{{\bf XYZ}}
\nc{\bY}{{\bf Y}}
\nc{\bF}{{\bf F}}
\nc{\bS}{{\bf S}}
\nc{\dV}{{\delta V}}
\nc{\dE}{{\delta E}}
\nc{\TT}{{\Theta}}
\nc{\dPsi}{{\delta\Psi}}
\nc{\order}{{\cal O}}
\nc{\Rout}{R_{\rm out}}
\nc{\eplus}{e_+}
\nc{\eminus}{e_-}
\nc{\epm}{e_\pm}
\nc{\eps}{\varepsilon}
\nc{\vnabla}{{\vec\nabla}}
\nc{\G}{\Gamma}
\nc{\w}{\omega}
\nc{\mh}{h}
\nc{\mg}{g}
\nc{\vphi}{\varphi}
\nc{\tlambda}{\tilde\lambda}
\nc{\be}{\begin{equation}}
\nc{\ee}{\end{equation}}
\nc{\ba}{\begin{eqnarray}}
\nc{\ea}{\end{eqnarray}}

\nc{\g}{\gamma}
\nc{\ol}{\overline}

\newtheorem{theorem}{Theorem}[section]
\newtheorem{lemma}[theorem]{Lemma}
\newtheorem{prop}[theorem]{Proposition}
\newtheorem{corollary}[theorem]{Corollary}
\newtheorem{defin}[theorem]{Definition}
\newtheorem{rmk}[theorem]{Remark}

\nc{\pT}{\partial_T}
\nc{\pz}{\partial_z}
\nc{\pt}{\partial_t}
\nc{\la}{\langle}
\nc{\ra}{\rangle}
\nc{\infint}{\int_{-\infty}^{\infty}}
\nc{\halfwidth}{6.5cm}
\nc{\figwidth}{10cm}
\newcommand{\f}{\frac}

\nc{\nlayers}{L} \nc{\nsectors}{M}
\nc{\indicator}{\mathbf{1}}
\nc{\Rhole}{R_{\rm hole}}
\nc{\Rring}{R_{\rm ring}}
\nc{\neff}{n_{\rm eff}}
\nc{\Frem}{F_{\rm rem}}
\nc{\R}{\mathbb R}
\nc{\Z}{\mathbb Z}
\nc{\DD}{\Delta}
\nc{\cD}{\mathcal D}
\nc{\lnorm}{\left\|}
\nc{\rnorm}{\right\|}
\nc{\rnormp}{\right\|_{\ell^{p,\eps}}}
\nc{\rar}{\rightarrow}
\date{\today}
\begin{document}

\begin{abstract}

	We investigate $L^1(\R^n)\to L^\infty(\R^n)$ dispersive estimates for the   Schr\"odinger operator $H=-\Delta+V$ when there is an eigenvalue at zero energy and $n\geq 5$ is odd. In particular, we show  that if there is an eigenvalue at zero energy then there is a time dependent, rank one operator $F_t$ satisfying $\|F_t\|_{L^1\to L^\infty} \lesssim |t|^{2-\f n2}$ for $|t|>1$  such that
$$\|e^{itH}P_{ac}-F_t\|_{L^1\to L^\infty} \les |t|^{1-\f n2},\,\,\,\,\,\text{ for } |t|>1.$$
With stronger decay conditions on the potential it is possible to generate an operator-valued
expansion for the evolution, taking the form
\begin{align*}
	e^{itH} P_{ac}(H)=|t|^{2-\f n2}A_{-2}+
	|t|^{1-\f n2} A_{-1}+|t|^{-\f n2}A_0,
\end{align*}
with $A_{-2}$ and $A_{-1}$ finite rank operators
mapping $L^1(\R^n)$ to $L^\infty(\R^n)$ while $A_0$ maps weighted
$L^1$ spaces to weighted $L^\infty$ spaces.
The leading order terms $A_{-2}$ and $A_{-1}$ vanish when certain orthogonality conditions
between the potential $V$ and the zero energy eigenfunctions are satisfied.
We show that under the same orthogonality conditions, the  remaining  $|t|^{-\f n2}A_0$ term
also exists as a map from $L^1(\R^n)$ to $L^\infty(\R^n)$, hence $e^{itH}P_{ac}(H)$ satisfies
the same dispersive bounds as the free evolution despite
the eigenvalue at zero.

\end{abstract}

\title[Dispersive Estimates for Schr\"odinger Operators: Odd dimensions]{\textit{Dispersive Estimates for higher
dimensional Schr\"odinger Operators with threshold
eigenvalues I: The odd dimensional case}}

\author[M.~J. Goldberg, W.~R. Green]{Michael Goldberg and William~R. Green}

\address{Department of Mathematics \\
University of Cincinnati\\
Cincinnati, OH 45221-0025}
\email{Michael.Goldberg@uc.edu}
\address{Department of Mathematics\\
Rose-Hulman Institute of Technology \\
Terre Haute, IN 47803 U.S.A.}
\email{green@rose-hulman.edu}

\thanks{This  work  
was  partially  supported  by  a  grant  from  the  Simons  Foundation (Grant  Number 281057
to  the first author.)
The second author acknowledges the support of an AMS Simons Travel grant and a
Rose-Hulman summer professional development grant.}

\maketitle

\section{Introduction}

The free Schr\"odinger evolution on $\R^n$,
$$
	e^{-it\Delta}f(x)= \frac{1}{(4\pi i t)^{\frac{n}{2}}} \int_{\R^n} e^{-i|x-y|^2/4t}f(y)\, dy
$$
maps $L^1(\R^n)$ to $L^\infty(\R^n)$ with norm bounded by
$|4\pi t|^{-\f n2}$.  As an immediate consequence, solutions whose initial data belong
to $L^1(\R^n) \cap L^2(\R^n)$ experience time decay with respect to the supremum
norm even while a conservation law holds the $L^2$ norm constant.  While both the
dispersive bound and the conservation law can be verified with elementary Fourier analysis,
they act in concert to imply Strichartz estimates for the free Schr\"odinger equation,
which are not readily apparent in either the physical-space or frequency-space description of
the propagator.

The stability of dispersive estimates under pertubration by a short range potential,
that is for a Schr\"odinger operator of the form $H = -\Delta + V$, where $V$ is 
real-valued and decays at spatial infinity, is a well-studied problem.
Where possible,
the estimate is presented in the form
\begin{equation} \label{eq:dispersive}
\lnorm e^{itH}P_{ac}(H)\rnorm_{L^1(\R^n)\to L^\infty(\R^n)} \lesssim |t|^{-n/2}.
\end{equation}
Projection onto the continuous spectrum
of the spectrum of $H$ is needed as the perturbed Schr\"odinger operator
may possess  pure point spectrum that experiences no decay at large times.

The first results in this direction \cite{Rauch,JenKat,Jen,Mur,Jen2} studied mappings
between weighted $L^2(\R^n)$ in place of $L^1(\R^n)$ and $L^\infty(\R^n)$.
Estimates of type~\eqref{eq:dispersive} are proved in
\cite{JSS,Wed,RodSch,GS,Sc2,CCV,EG1,BecGol,Gr} in various dimensions, and with
different characterizations of the potential $V(x)$.
For a more detailed history, see the survey paper
\cite{Scs}.

If the potential satisfies a pointwise
bound $|V(x)| \les  (1+|x|)^{-\beta}$ for some $\beta > 1$, then  the
spectrum of $H$ is purely absolutely continuous on $(0,\infty)$, see
\cite[Theorem XIII.58]{RS1}.  This leaves
two principal areas of concern: a high-energy region when the spectral parameter $\lambda$ satisfies $\lambda>\lambda_1>0$ and a low-energy region $0<\lambda<\lambda_1$ for some fixed constant $\lambda_1>0$.  

It was
observed by the first author and Visan~\cite{GV} in dimensions $n \ge 4$,
that it is possible for the dispersive estimate to fail as $t \to 0$ if the potential is not sufficiently
differentiable even for a
bounded compactly supported potential.  
The failure of the dispersive estimate is a high energy phenomenon.  Positive results have been obtained
in dimensions $n=4,5$ by Cardoso, Cuevas, and Vodev  \cite{CCV}
using semi-classical techniques assuming that
$V$ has $\frac{n-3}{2}+\epsilon$ derivatives, and
by Erdo\smash{\u{g}}an and the second author in dimensions $n=5,7$,
\cite{EG1} under the assumption that $V$ is differentiable
up to order $\frac{n-3}{2}$, which is the optimal smoothness requirement in light of the counterexample
in \cite{GV}. 
The decay assumptions on the
potential in \cite{EG1} were later relaxed by the
second author in \cite{Gr}.
The much earlier result of Journ\'e, Soffer, 
Sogge~\cite{JSS} requires that $\widehat{V} \in L^1(\R^n)$ in lieu of a specific
number of derivatives.


In dimensions $n\leq 4$ in addition to zero energy
eigenvalues,
there is another class of obstructions at zero energy
called resonances.  Resonances are distributional solutions
of $H\psi=0$ with $\psi \notin L^2(\R^n)$ satisfying other dimension-specific criteria.
In dimension $n=1, 2$, a resonance occurs if $\psi \in L^p(\R^2)$ for some $p>2$, while
in dimensions $n=3,4$, the condition is instead
$(1+|x|)^{-\sigma} \psi \in L^2$ for certain values
of $\sigma>0$.  Working in dimension $n \geq 5$ ensures that $(-\Delta)^{-1}V$ is a bounded operator on $L^2(\R^n)$,
and forces zero to be an eigenvalue of $H$ if it is not a regular point of the spectrum.
In order to handle the low-energy contribution, one typically assumes that zero is
a regular point of the spectrum of $H$.  Our goal in this paper is to characterize the
evolution in odd dimensions $n \ge 5$ when that assumption fails.  The analogous case in even dimensions $n\geq 6$ is treated in a separate work~\cite{GGeven}.

It is known that in general obstructions at zero lead to a loss of time decay in the dispersive estimate.
Jensen and Kato~\cite{JenKat} showed that in three dimensions, if there is a resonance at zero energy then
the propagator $e^{itH}P_{ac}(H)$ (as an operator between polynomially weighted $L^2(\R^3)$ spaces) has leading order decay of $|t|^{-\f 12}$
instead of $|t|^{-\f 32}$.  The same effect occurs if zero is an eigenvalue, despite the fact that $P_{ac}(H)$
explicitly projects away from the associated eigenspace.
For all $n \geq 5$,
Jensen~\cite{Jen}
obtains leading order decay at the rate $|t|^{2-\frac{n}{2}}$ as an operator on weighted $L^2(\R^n)$ spaces if zero is an eigenvalue.  The
subsequent terms of the asymptotic expansion have decay rates $|t|^{1-\frac{n}{2}}$ and $|t|^{-\frac{n}{2}}$
and map between more heavily weighted $L^2(\R^n)$ spaces.  
We are able to recover the same structure
of time decay with respect to mappings from $L^1(\R^n)$ to $L^\infty(\R^n)$, with a finite rank leading order
term and a remainder that belongs to weighted spaces.  Our results imply Jensen's results by embedding
$(1+|x|)^{-\sigma}L^2(\R^n) \subset L^1(\R^n)$ and $L^\infty(\R^n) \subset (1+|x|)^{\sigma}L^2(\R^n)$ for $\sigma > \f n2$,
and in fact our results imply $L^2$ estimates with
smaller weights than required in \cite{Jen}.

Perhaps the most novel result we prove is 
that~\eqref{eq:dispersive} is once again satisfied, with no weights, provided the zero-energy eigenfunctions
satisfy the two orthogonality conditions stated in Theorem~\ref{thm:main}  part (\ref{thmpart3}).

The orthogonality conditions mentioned above are directly tied to the spatial decay of eigenfunctions solving $H\psi = 0$.
For a ``generic'' eigenfunction, $|\psi(x)| \sim |x|^{2-n}$ for large $x$, a property inherited from the Green's function of
the Laplacian.  Vanishing moments of $V\psi$ are associated with faster decay of $\psi(x)$, more specifically
the condition $\int_{\R^n}V\psi\,dx = 0$ in Theorem~\ref{thm:main}
implies that $|\psi(x)| \les |x|^{1-n}$, and if $\int_{\R^n} x_jV\psi\,dx = 0$ as well for each $j = 1\ldots n$
then in fact $|\psi(x)| \les |x|^{-n}$.
At the same time these conditions also cause some leading-order terms of the expansion in~\cite{Jen} to vanish.
We show that a more subtle cancellation takes place in the remainder terms, which is why they can be stated with
reduced weights as in Theorem~\ref{thm:main} part (\ref{thmpart2}), or no weights as in part (\ref{thmpart3})
of the same theorem.  
The first author proved a similar result in three dimensions under the condition $\psi \in L^1(\R^3)$, \cite{goldE} in place of the orthogonality conditions.  

We emphasize that the threshold spectral properties of a given
Schr\"odinger operator $H = -\Delta + V$ are difficult to discern from examing the potential $V$ alone.
Eigenvalues at zero can be ruled out if the potential has a lower bound
$V(x) \geq -C_0(1+|x|^2)^{-1}$, where the constant $C_0$ can be determined
from the corresponding Hardy inequality. 
However for potentials with large negative part there is no known simple test
to determine whether an eigenvalue is present at zero.

At the end of Section~\ref{sec:Spec} we describe a large family of 
compactly supported potentials
that possess a zero-energy eigenfunction with prescribed polynomial decay at infinity. 
We believe that the eigenspace is typically one-dimensional, in which case
the orthogonality conditions of Theorem~\ref{thm:main} are satisfied.
The construction and supporting arguments are adapted from~\cite{goldE}.

To state our main results,
first choose a smooth cut-off function $\chi(\lambda)$
with $\chi(\lambda)=1$ if $\lambda<\lambda_1/2$ and
$\chi(\lambda)=0$ if $\lambda>\lambda_1$, for a
sufficiently small $0<\lambda_1\ll 1$.
In addition, we use the notation $\la x\ra := (1+|x|)$, and define the
(polynomially)
weighted $L^p$ spaces
\begin{equation*}
\lnorm f\rnorm_{L^{p,\sigma}} := \lnorm \la x\ra^\sigma f\rnorm_p
\end{equation*}
and the abbreviations $a-:=a-\epsilon$ and $a+:=a+\epsilon$ for a small, but
fixed, $\epsilon>0$.

We prove the following low energy bounds.

\begin{theorem}\label{thm:reg}

	Assume that $n \geq 5$ is odd, 
	$|V(x)|\les \la x\ra ^{-\beta}$ for some $\beta>n$
	and that zero is  not an eigenvalue of
	$H=-\Delta+V$ on $\R^n$.  Then,
	$$
		\|e^{itH} \chi(H)P_{ac}(H)\|_{L^1\to L^\infty}
		\les |t|^{-\f n2}.
	$$
	
\end{theorem}

\begin{theorem}\label{thm:main}

	Assume that $n \geq 5$ is odd, $|V(x)|\les \la x\ra ^{-\beta}$, and that zero is an eigenvalue of
	$H=-\Delta+V$ on $\R^n$.  The low energy Schr\"odinger propagator $e^{itH}\chi(H)P_{ac}(H)$
	possesses the following structure:
	\begin{enumerate}
	\item \label{thmpart1} Suppose that
	there exists $\psi \in {\rm Null}\,H$ such that 
	$\int_{\R^n} V\psi\,dx \not= 0$.  Then there is a rank one time dependent operator $\|F_t\|_{L^1 \to L^\infty}\les
    |t|^{2-\f n2}$ such that for $|t| > 1,$
	$$
	e^{itH} \chi(H)P_{ac}(H) - 
	F_t= \mathcal E_1(t).
	$$ 
	Where  $\|\mathcal E_1 \|_{L^1 \to L^\infty}=o(|t|^{2-\f n2})$ if $\beta>n$
	and  $\|\mathcal E_1 \|_{L^1 \to L^\infty}=O(|t|^{1-\f n2})$
	if $\beta>n+4$.
	\item \label{thmpart2} Suppose 
	that $\int_{\R^n} V\psi\, dx = 0$ for each $\psi \in {\rm Null}\,H$ but
           $\int_{\R^n} x_j V \psi\, dx \not= 0$ for some $\psi$ and some $j \in [1, \ldots, n]$.
	 Then there exists a finite rank time dependent operator $G_t$ satisfying
	$\lnorm G_t\rnorm_{L^1 \to L^\infty} \les |t|^{1-\frac{n}{2}}$ such that for $|t|>1$,
	$$
	e^{itH} \chi(H)P_{ac}(H) - G_t = \mathcal E_2(t).
	$$
	Where  $\|\mathcal E_2 \|_{L^1 \to L^\infty}=O(|t|^{1-\f n2})$ and
	$ \|\mathcal E_2 \|_{ L^{1,0+} \to L^{\infty,0-}}=o(|t|^{1-\f n2})$
	 if $\beta>n+4$
	and  $\|\mathcal E_2 \|_{L^{1,1} \to L^{\infty,-1}}=O(|t|^{-\f n2})$
	if $\beta>n+8$.
	\item \label{thmpart3} Suppose $\beta>n+8$ and
	that $\int_{R^n} V\psi\, dx = 0$ and $\int_{R^n} x_j V \psi\, dx = 0$
	for all $\psi \in {\rm Null}\,H$ and all $j \in [1,\ldots,n]$.  Then
	$$
	\lnorm e^{itH} \chi(H)P_{ac}(H)\rnorm_{L^1\to L^\infty} \les |t|^{-\frac{n}{2}}.
	$$
	\end{enumerate}
	
\end{theorem}

We note that the assumption that
$\int_{\R^n} V\psi\, dx = 0$ for each $\psi \in {\rm Null}\,H$ is equivalent to assuming that the operator
$P_eV1=0$ with $P_e$ the projection onto the zero-energy
eigenspace.  Further, $\int_{R^n} x_j V \psi\, dx = 0$
for each $j=1,2,\dots, n$
is equivalent to assuming the operator $P_eVx=0$.

These results are fashioned similarly to the asymptotic expansions in~\cite{Jen}, with particular
emphasis on the behavior of the resolvent of $H$ at low energy.  If one assumes greater decay
of the potential, then it becomes possible to carry out the resolvent expansion to a greater number of
terms, which permits a more detailed description of the time decay of $e^{itH}\chi(H)P_{ac}(H)$.
We note that while $F_t$ and $G_t$ above have a concise construction, expressions for higher order
terms in the expansion are unwieldy enough to discourage writing out an exact formula.

The power series statement of the main theorem is as follows. 
\begin{corollary}\label{cor:ugly}

If $|V(x)|\les \la x\ra^{-n-8-}$, and there is an
eigenvalue of $H$ at zero energy, then 
we have the operator-valued expansion
\begin{align} \label{eq:corollary}
	e^{itH} \chi(H)P_{ac}(H)=C_n|t|^{2-\f n2}P_eV1VP_e +
	|t|^{1-\f n2} A_{-1} + |t|^{-\f n2}A_0(t).
\end{align}
There exist uniform bounds for 
$P_eV1VP_e:L^1\to L^\infty$, 
$A_{-1}:L^{1}\to L^{\infty}$, and $A_0(t):L^{1,2}\to L^{\infty, -2}$.  
The operator $P_eV1VP_e$ is a rank one projection
and $A_{-1}$ is finite rank.
Furthermore, if $P_eV1=0$, then \
$A_0(t):L^{1,1}\to L^{\infty, -1}$. If $P_eV1=0$ and $P_eVx=0$
then $A_{-1}$ vanishes and 
$A_0(t):L^1\to L^\infty$ uniformly in $t$.  

\end{corollary}

We note that this expansion could continue indefinitely
in powers of $|t|^{-\f n2 -k}$, $k\in \mathbb N$.  The
operators would be finite rank 
between successively more heavily
weighted spaces and  would require more decay on the
potential $V$, we do not pursue this issue.

Global dispersive estimates are known in all lower dimensions when zero is
not a regular point of the spectrum.  They are due to the first author and Schlag~\cite{GS}
in one dimension, the second author and Erdo\smash{\u{g}}an~\cite{EG} in two dimensions,
Yajima~\cite{Yaj3} and Erdo\smash{\u{g}}an-Schlag~\cite{ES2,ES} in three dimensions, and to Erdo\smash{\u{g}}an
and the authors~\cite{EGG} in four dimensions.  Except for the last of these, the
low-energy argument builds upon the series expansion for resolvents set forth
in~\cite{JenKat, Jen2, JN}.  We continue to follow this line of argument and work
with high-dimensional resolvent expansions similar to
those in~\cite{Jen}. 

We note that the estimates we prove can be combined with
the large energy estimates in, for example, \cite{Yaj,FY} to 
prove analogous statements for the full evolution
$e^{itH}P_{ac}(H)$ without the low-energy cut-off.
This requires more assumptions on the
the potential, that its polynomially weighted Fourier transform satisfies
$$
	\mathcal F(\la x\ra^{2\sigma}V)\in L^{n_*}(\R^n)
	\qquad \textrm{ for } \sigma>\frac{1}{n_*}=\frac{n-2}{n-1}.
$$
Roughly speaking, this corresponds to having more than
$\frac{n-3}{2}+\frac{n-3}{n-2}$ derivatives of $V$
in $L^2$.

In addition there has been work on the $L^p$ boundedness
of the wave operators, which are defined by strong limits
on $L^2(\R^n)$,
$$
	W_{\pm}=s\mbox{-}\lim_{t\to \pm\infty}
	e^{itH}e^{it\Delta}.
$$
The $L^p$ boundedness of the wave operators is of
interest to our line of inquiry because of the  
`intertwining property'
$$
	f(H)P_{ac}=W_{\pm}f(-\Delta)W_{\pm}^*,
$$
which is valid for Borel functions $f$. 
In dimensions $n\geq 5$, boundedness of the 
wave operators on $L^p$ for $\frac{n}{n-2}<p<\frac{n}{2}$
in the presence of an eigenvalue at zero
was established by Yajima~\cite{Yaj} in odd dimensions,
and Finco-Yajima \cite{FY} in even dimensions.
In particular, these results imply the mapping estimate
$$
	\|e^{itH}P_{ac}(H)\|_{L^p\to L^{p\prime}}\les
	|t|^{-\f n2+\frac{n}{p}}.
$$
Here $p^\prime$ is the conjugate exponent satisfying
$\frac{1}{p}+\frac{1}{p^\prime}=1$.  Roughly speaking,
the wave operator results yield a time decay rate of
$|t|^{-\frac{n}{2}+2+}$.  Similar results in lower dimensions can be found in \cite{Yaj3,JY4}.

The fact that $n \geq 5$ allows for greater uniformity of approach, as there are no
special dimension-specific considerations related to the distinction between resonances
and eigenvalues at zero.  There are, however, significant differences in the low-energy
expansion of the resolvent depending on whether $n$ is even or odd.  While
the dispersive bounds stated in Theorem~\ref{thm:main} hold for all $n \geq 5$,
there are not enough shared elements in the computation to treat the even and odd
dimensional cases side by side.   The present paper considers odd $n$.  The case of even $n$ is 
more technically challenging due to the logarithmic
behavior of the resolvent operators and is
considered
in the companion paper, \cite{GGeven}.

We define the limiting resolvent operators
\begin{align*}
	R_V^{\pm}(\lambda^2)=\lim_{\epsilon \to 0^+}
	(-\Delta+V-(\lambda^2\pm i \epsilon))^{-1}.
\end{align*}
These operators are well-defined on certain weighted
$L^2(\R^n)$ spaces, see \cite{agmon}.  In fact, 
there is a zero energy eigenvalue precisely when this
operator becomes unbounded as $\lambda \to 0$.

As usual (cf. \cite{RodSch,GS,Sc2}), the dispersive estimates follow by considering the operator $e^{itH}\chi(H)P_{ac}(H)$
as an element of the functional calculus of $H$.  Using the Stone formula, we have
\begin{align*}
	e^{itH}\chi(H)P_{ac}(H)f(x)=\frac{1}{2\pi i} 
	\int_0^\infty e^{it\lambda^2} \lambda \chi(\lambda)
	[R_V^+(\lambda^2)-R_V^-(\lambda^2)]f(x)\, 
	d\lambda,
\end{align*}
with the difference of resolvents $R_V^{\pm}(\lambda)$ providing the absolutely continuous spectral measure.
For $\lambda > 0$ (and if also at $\lambda = 0$ if zero is a regular point of the spectrum) the resolvents are well-defined on certain
weighted $L^2$ spaces.  The key issue when zero energy is not regular is to control
the singularities in the spectral measure  as $\lambda\to 0$.

Here $R_V^\pm(\lambda^2)$ are operators whose integral
kernel we write as $R_V^\pm(\lambda^2)(x,y)$.  That is,
the action of the operator is defined by
\begin{align*}
	R_V^\pm(\lambda^2)f(x)= \int_{\R^n}R_V^\pm(\lambda^2)(x,y) f(y) \, dy.
\end{align*}

The analysis in this paper focuses on bounding the
oscillatory integral
\begin{align}\label{Stone}
	\int_0^\infty e^{it\lambda^2} \lambda \chi(\lambda)
	[R_V^+(\lambda^2)-R_V^-(\lambda^2)](x,y)\, 
	d\lambda
\end{align}
in terms of $x,y$ and $t$.  A uniform bound of the form
$\sup_{x,y} |\eqref{Stone}|\les |t|^{-\alpha}$ would give
us an estimate on $e^{itH}P_{ac}(H)$ as an operator from
$L^1\to L^\infty$.   We leave open the option of dependence
on $x$ and $y$ to allow for estimates between weighted
$L^1$ and weighted $L^\infty$ spaces.
That is, an estimate of
the form $|\eqref{Stone}|\les |t|^{-\alpha}\la x\ra^{\sigma} \la y\ra^{\sigma'}$ yields an estimate
as an operator from $L^{1,\sigma'}\to L^{\infty, -\sigma}$.

Accordingly, we study expansions for the resolvent operators $R_V^{\pm}(\lambda^2)$ in a neighborhood
of zero.  The type of terms present is heavily influenced by whether $n$ is even or odd.
In odd dimensions the expansion is a formal Laurent series 
$$
	R_V^\pm(\lambda^2) = A\lambda^{-2} 
	+ B\lambda^{-1} + O(1)
$$ with operator-valued coefficients.  The
operators $A$ and $B$ are zero if there are no
zero-energy eigenvalues (or resonances in dimensions
$n=1,3$).
In even
dimensions the expansions are more complicated, involving terms of the
form $\lambda^k (\log \lambda)^\ell$, $k \geq -2$, see
for example \cite{Jen,Jen2,EG,EGG,GGeven}.

The organization of the paper is as follows.
We begin in Section~\ref{sec:resolv} by developing expansions for the free
resolvent and develop necessary machinery to
understand the spectral measure
$E'(\lambda)=\frac{1}{2\pi i}
[R_V^+(\lambda^2)-R_V^-(\lambda^2)]$.  In Section~\ref{sec:disp}
we prove dispersive estimates for the tail of the 
Born series, \eqref{eq:bstail}, which is the portion of the
evolution that is sensitive to the existence of zero-energy eigenvalues.  Next, in Section~\ref{sec:finitebs},
we prove dispersive estimates for the finite Born
series series, \eqref{eq:finitebs}, which is the
portion of the low energy evolution
that is unaffected by zero-energy eigenvalues.
Collectively these form a proof of  Theorem~\ref{thm:main}.
In Section~\ref{sec:Spec} we provide a characterization
of the spectral subspaces of $L^2$ related to the
zero energy eigenspace.  Finally, Section~\ref{sec:ests}
contains an index of technical integral estimates that arise in the course
of the preceding calculations.

\section{Resolvent Expansions Around Zero}\label{sec:resolv}

In this section we first develop expansions for the integral kernels of the free
resolvents $R_0^{\pm}(\lambda^2):=\lim_{\epsilon\to 0^+}
(-\Delta-(\lambda^2\pm i\epsilon))^{-1}$ to understand the pertubed resolvent operators
$R_V^{\pm}(\lambda^2):=\lim_{\epsilon\to 0^+}
(-\Delta+V-(\lambda^2\pm i\epsilon))^{-1}$
with the aim of understanding the spectral measure in
\eqref{Stone}.

In developing these expansions we employ the following
notation
$$
	f(\lambda)=\widetilde O(g(\lambda))
$$
to indicate that
$$
	\frac{d^j}{d\lambda^j}f(\lambda)=O\bigg(\frac{d^j}{d\lambda^j}g(\lambda)\bigg).
$$
If the relationship holds only for the first $k$ 
derivatives, we use the notation 
$f(\lambda)=\widetilde{O}_k(g(\lambda))$.  With a slight
abuse of notation, we may write $f(\lambda)=\widetilde O(\lambda^k)$ for an integer $k$, to indicate that $\frac{d^j}{d\lambda^j}f(\lambda)=O(\lambda^{k-j})$.
This distinction is particularly important for when
$k\geq 0$ and $j>k$.

Writing the free resolvent kernel in terms of the Hankel functions we have
\begin{align}\label{Hankel}
	R_0(z)(x,y)=\frac{i}{4} \bigg(\frac{z^{1/2}}{2\pi |x-y|}\bigg)^{\frac{n}{2}-1} H_{\frac{n}{2}-1}^{(1)}(z^{1/2}|x-y|).
\end{align}
Here $H_{\frac{n}{2}-1}^{(1)}(\cdot)$ is the Hankel function of the first kind.   Since $n$ is odd, these are
Hankel functions of half-integer order, which can be expressed in closed form.
We use the following explicit representation for the kernel of the limiting resolvent operators $R_0^\pm(\lambda^2)$ 
(see, {\it e.g.}, \cite{Jen})  
$$R_0^\pm(\lambda^2)(x,y)=\mathcal G_n(\pm\lambda,|x-y|),$$ 
where
\begin{align}\label{Gn poly}
   \mathcal G_n(\lambda,r)=C_n \frac{e^{i\lambda r}}{r^{n-2}}
    \sum_{\ell=0}^{\frac{n-3}{2}} \frac{(n-3-\ell)!}
     {\ell!(\frac{n-3}{2}-\ell)!}
   (-2i r \lambda)^\ell.   
\end{align} 
For small $\lambda$, one can expand these in a 
Taylor series, as in Lemma~3.5 of \cite{Jen} to see with $G_j(x,y)=c_j |x-y|^{2+j-n}$ with $c_j$ real-valued 
constants.
\begin{align}\label{Taylor sloppy}
	\mathcal G_n(\lambda,r)=&G_0+\sum_{j=1}^{\frac{n-3}{2}}
	\lambda^{2j}G_{2j}+i\lambda^{n-2} G_{n-2}+\lambda^{n-1}G_{n-1}
	+i\lambda^n G_n
	+ E(\lambda),
	\textrm{ as } \lambda \to 0.
\end{align}
Where the error term $E(\lambda)=O(\lambda^{n+1})$ is
understood as a Hilbert-Schmidt operator between
weighted $L^2$ spaces.  We can (and need to) be more
delicate with this error term.

It is quite natural to view
this as an operator between weighted $L^2$ spaces as
$G_0$ is a scalar multiple of the fractional integral
operator $I_2$ whereas the remaining terms are either
scalar multiples of the fractional integral operators
$I_{2j+2}$ or can be
bounded in a Hilbert-Schmidt norm with sufficiently
large polynomial weights.  In particular,
we note that
\begin{align}
	G_0(x,y)&=c_0 |x-y|^{2-n}=(-\Delta)^{-1}(x,y),\label{G0 defn}\\
	G_{n-2}(x,y)&=c_{n-2},  \label {Gn-2 defn}\\	
	G_{n}(x,y)&=c_{n}|x-y|^2=c_{n}(x-y)\cdot(x-y)
	=c_n[|x|^2-2x\cdot y+|y|^2], \label {Gn defn}
\end{align}
We may also use the notation $G_{n-2}=c_{n-2} 1$, where
$1$ indicates the operator with kernel
$1(x,y)=1$.

\begin{lemma}\label{lem:R0exp}

	For $\lambda\leq \lambda_1$,
	we have the expansion(s) for the free resolvent,
	\begin{align*}
		R_0^{\pm}(\lambda^2)&=G_0
		+\sum_{j=1}^{\frac{n-3}{2}}
		\lambda^{2j}G_{2j}\pm i\lambda^{n-2} G_{n-2}
		+E_0^{\pm}(\lambda)\\
		&=G_0+\sum_{j=1}^{\frac{n-3}{2}}
		\lambda^{2j}G_{2j}
		\pm i\lambda^{n-2} G_{n-2}+\lambda^{n-1}G_{n-1}
		+E_1^\pm(\lambda)\\
		&=G_0+\sum_{j=1}^{\frac{n-3}{2}}
		\lambda^{2j}G_{2j}
		\pm i\lambda^{n-2} G_{n-2}+\lambda^{n-1}G_{n-1}
		 \pm i\lambda^n G_n 
		+E^\pm_2(\lambda) \\
		&=G_0+\sum_{j=1}^{\frac{n-3}{2}}
		\lambda^{2j}G_{2j}
		\pm i\lambda^{n-2} G_{n-2}+\lambda^{n-1}G_{n-1}
		 \pm i\lambda^n G_n \begin{aligned}[t]&+ \lambda^{n+1}G_{n+1} \\
		&\pm i\lambda^{n+2}G_{n+2}  
		+E^\pm_3(\lambda). \end{aligned}
	\end{align*}
	Where, for any $0\leq \ell\leq 1$,
	\begin{align*}
		E_0^{\pm}(\lambda)&=|x-y|^\ell
		\widetilde O_{\frac{n-1}{2}}
		(\lambda^{n-2+\ell}), \\
		E_1^{\pm}(\lambda) 
		&=|x-y|^{1+\ell}\widetilde O_{\frac{n+1}{2}}(\lambda^{n-1+\ell}), \\
		\quad
		E_2^{\pm}(\lambda)&=|x-y|^{2+\ell} \widetilde O_{\frac{n+3}{2}}(\lambda^{n+\ell}),\\
		E^\pm_3(\lambda) &= |x-y|^{4+\ell}\widetilde O_{\frac{n+7}{2}}(\lambda^{n+2+\ell}).
	\end{align*}

\end{lemma}

\begin{proof}

Using a Taylor series expansion on \eqref{Gn poly}
when $\lambda |x-y|\leq \frac{1}{2}$, one has
\begin{align*}
	\mathcal G_n(\lambda,r)&=G_0+\sum_{j=1}^{\frac{n-3}{2}}
	\lambda^{2j}G_{2j}
	+i\lambda^{n-2} G_{n-2}+\lambda^{n-1}G_{n-1}
	+i\lambda^n G_n
	+\lambda^{n+1}G_{n+1}\\
	&+i\lambda^{n+2}G_{n+2}+
	\widetilde O(\lambda^{n-2}(\lambda |x-y|)^{5}).
\end{align*}
This expansion can be truncated earlier, using 
$G_j(x,y)=c_j |x-y|^{n-2+j}$ and $\lambda |x-y|\les 1$.

On the other hand, if $\lambda|x-y|\gtrsim 1$ we note
that differentiation in $\lambda$ in \eqref{Gn poly}
is comparable to either division by $\lambda$ or multiplication by $|x-y|$.  So that, for
$0\leq k\leq \frac{n-1}{2}$,
\begin{align*}
	|\partial_\lambda^k R_0^{\pm}(\lambda^2)(x,y)|
	&\les \bigg(\frac{1}{|x-y|^{n-2}}
	+\frac{\lambda^{\frac{n-3}{2}}}{|x-y|^{\frac{n-1}{2}}}
	\bigg)(\lambda^{-k}+|x-y|^k)\les \lambda^{n-2-k}.
\end{align*}
Here we use that $|x-y|^{-1}\les \lambda$.  Using that
$\lambda|x-y|\gtrsim 1$, can can gain more $\lambda$
smallness at the cost of growth in $|x-y|$.  Specifically,
for any $\ell\geq 0$ and $0\leq k\leq \frac{n-1}{2}$,
\begin{align*}
	|\partial_\lambda^k R_0^{\pm}(\lambda^2)(x,y)|
	\les \lambda^{n-2-k}(\lambda|x-y|)^\ell.
\end{align*}

For the first expansion we take
$$
	E_0^{\pm}(\lambda)=R_0^{\pm}(\lambda^2)-G_0-
	\sum_{j=1}^{\frac{n-3}{2}}
	\lambda^{2j}G_{2j} \mp i\lambda^{n-2}G_{n-2}.
$$
The error bounds if $\lambda |x-y|\leq \frac{1}{2}$ are
clear from expanding \eqref{Gn poly} in a Taylor series. 
On the other hand, if  $\lambda |x-y|\gtrsim 1$ we have that
$$
	\bigg|\partial^{k}_\lambda \big(G_0+
		\sum_{j=1}^{\frac{n-3}{2}}
		\lambda^{2j}G_{2j}\mp i\lambda^{n-2}G_{n-2}
	\big)\bigg|
	\les \lambda^{n-2-k}.
$$
Where one uses that $|x-y|^{-1}\les \lambda$ in this
case.
The remaining error terms $E_1^\pm(\lambda)$ and
$E_2^\pm(\lambda)$ arise from subtracting more Taylor
terms from the free resolvent.  
$$
	E_1^{\pm}(\lambda)=R_0^{\pm}(\lambda^2)-G_0-
	\sum_{j=1}^{\frac{n-3}{2}}
	\lambda^{2j}G_{2j}\mp i\lambda^{n-2} G_{n-2}-\lambda^{n-1}G_{n-1},
$$
\begin{equation*}
	E_2^{\pm}(\lambda)=R_0^{\pm}(\lambda^2)-G_0-
	\sum_{j=1}^{\frac{n-3}{2}}
	\lambda^{2j}G_{2j}\mp i\lambda^{n-2} 
	G_{n-2}-\lambda^{n-1}G_{n-1}
	\mp i\lambda^n G_n,
\end{equation*}
and
\begin{align*}
E_3^{\pm}(\lambda)=R_0^{\pm}(\lambda^2)-G_0-
	\sum_{j=1}^{\frac{n-3}{2}}
	\lambda^{2j}G_{2j}\mp i\lambda^{n-2} 
	G_{n-2}&-\lambda^{n-1}G_{n-1}
	\mp i\lambda^n G_n \\
	&-\lambda^{n+1}G_{n+1} \mp i\lambda^{n+2}G_{n+2}.
\end{align*}

In particular, we note
that for $k\leq n-2+j$,
$$
	\big| \partial_\lambda^k \lambda^{n-2+j}G_{n-2+j}
	\big| \les \lambda^{n-2+j-k}G_{n-2+j}\les
	\lambda^{n-2+j-k}|x-y|^{j}.
$$
With the bound being zero if $k>n-2+j$.
As before, the bounds are clear for $\lambda|x-y|\leq \f12$, for $\lambda|x-y|\gtrsim 1$, one has for
any $\ell\geq 0$
$$
	\big| \partial_\lambda^k \lambda^{n-2+j}G_{n-2+j}
	\big| \les
	\lambda^{n-2+j-k}|x-y|^{j}(\lambda|x-y|)^\ell.
$$
Finally, we note that if $\lambda|x-y|\gtrsim 1$ and
$k\geq \frac{n+3}{2}$ when differentiating $R_0^\pm(\lambda^2)$,
multiplication by $|x-y|$ dominates
division by $\lambda$, thus 
we have
\begin{align*}
	|\partial_\lambda^k R_0^{\pm}(\lambda^2)(x,y)|
	\les \frac{\lambda^{\frac{n-3}{2}}}{|x-y|^{\frac{n-1}{2}}}
	(\lambda^{-k}+|x-y|^k)
	\les \lambda^{\frac{n-3}{2}} |x-y|^{k-\frac{n-1}{2}}
	(\lambda|x-y|)^\ell.
\end{align*}
This suffices to prove the bounds for $E_j^\pm(\lambda)$.

\end{proof}

It is important to note that the kernels of the operators
$G_j$ are strictly real-valued.  
The hypotheses of the lemma below are not optimal, but
suffice for our purposes.  We give the proof of this Lemma in Section~\ref{sec:ests}.

\begin{lemma}\label{lem:iterated}

	If $|V(x)|\les \la x\ra^{-\frac{n+1}{2}-}$, $\sigma>\f12$ and 
	$\kappa\geq \frac{n-3}{4}$, then
	\begin{align*}
		\| (R_0^\pm(\lambda)^2V)^{\kappa-1}(y,\cdot)R_0(\cdot, x)\|_{L^{2,-\sigma}_y}
		\les \la \lambda \ra^{\kappa(\frac{n-3}{2})}.
	\end{align*}
	uniformly in $x$.

\end{lemma}

Define the operator $U:L^2\to L^2$ by its  kernel
$$
	U(x)=\left\{
		\begin{array}{ll}
		1 & V(x)\geq 0\\
		-1 & V(x)<0
		\end{array}
	\right.
$$
That is, $U$ is the sign of $V$.
Further define $v=|V|^{\f12}$, $w=Uv$ and
$$
M^{\pm}(\lambda)=U+vR_0^{\pm}(\lambda^2)v.
$$
We wish to use the symmetric resolvent identity, 
\begin{align}\label{symmresid}
	R_V^{\pm}(\lambda^2)=R_0^{\pm}(\lambda^2)-
	R_0^{\pm}(\lambda^2)vM^{\pm}
	(\lambda)^{-1}vR_0^{\pm}(\lambda^2),
\end{align}
which is valid for $\Im(\lambda)>0$,
to understand the spectral measure in \eqref{Stone}.

Lemma~\ref{lem:iterated} allows us to make sense of
the symmetric
resolvent identity, provided 
$|V(x)|\les \la x\ra^{-\frac{n+1}{2}-}$,
by iterating the standard resolvent identity
$$
	R_V^{\pm}(\lambda^2)=R_0^{\pm}(\lambda^2)-
	R_0^{\pm}(\lambda^2)VR_V^{\pm}(\lambda^2)
	=R_0^{\pm}(\lambda^2)-
	R_V^{\pm}(\lambda^2)VR_0^{\pm}(\lambda^2)
$$
on both sides of $M^{\pm}(\lambda)^{-1}$ in 
\eqref{symmresid} a sufficient number of times to get to a local $L^2$ space, from
which multiplication by $v$ takes the iterated resolvents
to $L^2$.

Our main tool used to invert 
$M^\pm(\lambda)=U+vR_0^\pm(\lambda^2)v$  
for small $\lambda$, is the
following  lemma (see Lemma 2.1  in \cite{JN}).
\begin{lemma}\label{JNlemma}
Let $A$ be a closed operator on a Hilbert space $\mathcal{H}$ and $S$ a projection. Suppose $A+S$ has a bounded
inverse. Then $A$ has a bounded inverse if and only if
$$
B:=S-S(A+S)^{-1}S
$$
has a bounded inverse in $S\mathcal{H}$, and in this case
$$
A^{-1}=(A+S)^{-1}+(A+S)^{-1}SB^{-1}S(A+S)^{-1}.
$$
\end{lemma}

Following the terminology used in \cite{Sc2,EG,EGG},

\begin{defin}
	We say an operator $K:L^2(\R^n)\to L^2(\R^n)$ with kernel
	$K(\cdot,\cdot)$ is absolutely bounded if the operator with kernel
	$|K(\cdot,\cdot)|$ is bounded from $L^2(\R^n)$ to $L^2(\R^n)$.
\end{defin}

We recall the definition of the Hilbert-Schmidt norm
of an operator $K$ with integral kernel
$K(x,y)$ ,
\begin{align*}
	\|K\|_{HS}=\bigg(\iint_{\R^{2n}} |K(x,y)|^2\, dx\, dy
	\bigg)^{\f12}.
\end{align*}
We note that Hilbert-Schmidt and finite rank operators
are immediately absolutely bounded.

\begin{lemma}\label{lem:Mexp}

	Assuming that $v(x)\les \la x\ra^{-\beta}$.
	If $\beta>\frac{n}{2}+\ell$, then for
	$0\leq \ell\leq 1$, we have
	\begin{align}
		M^{\pm}(\lambda)
		&=U+vG_0v+
		\sum_{j=1}^{\frac{n-3}{2}}
		\lambda^{2j}vG_{2j}v\pm i\lambda^{n-2}vG_{n-2}v
		+M_0^\pm(\lambda),\label{M0exp}
	\end{align}
	Where  the kernels of the
	operators $G_j$ are absolutely bounded with 
	real-valued kernels.
	Further,
	\begin{align}\label{M0err}
		\sum_{j=0}^{\frac{n-1}{2}} \| 
		\sup_{0<\lambda<\lambda_1} \lambda^{j+2-n-\ell}
		\partial_\lambda^j M_0^{\pm}(\lambda)\|_{HS}
		\les 1.
	\end{align}
	If $\beta>\frac{n}{2}+2+\ell$, then
	\begin{align}\label{M0exp2}
		M_0^{\pm}(\lambda)=
		\lambda^{n-1}
		vG_{n-1}v\pm i\lambda^n v G_n v+M_1^{\pm}(\lambda),
	\end{align}
	with
	\begin{align}\label{M1err}
		\sum_{j=0}^{\frac{n+1}{2}} \| 
		\sup_{0<\lambda<\lambda_1} \lambda^{j-n-\ell}
		\partial_\lambda^j M_1^{\pm}(\lambda)\|_{HS}
		\les 1.
	\end{align}
	If $\beta>\frac{n}{2}+4+\ell$, then
	\begin{align}\label{M1exp}
		M_1^{\pm}(\lambda)=
		\lambda^{n+1}v G_{n+1}v\pm i\lambda^{n+2}v
		G_{n+2}v
		+M_2^{\pm}(\lambda)
	\end{align}
	with
	\begin{align}\label{M2err}
		\sum_{j=0}^{\frac{n+1}{2}} \| 
		\sup_{0<\lambda<\lambda_1} \lambda^{j-n-2-\ell}
		\partial_\lambda^j M_2^{\pm}(\lambda)\|_{HS}
		\les 1.
	\end{align}

\end{lemma}

\begin{proof}

	The proof follows from the definition of the operators
	$M^{\pm}(\lambda)$ and the expansion for the free
	resolvent in Lemma~\ref{lem:R0exp}.  The bound on the
	error terms follows from the fact 
	that if $k>-\frac{n}{2}$ then
	$\la x\ra^{-\beta}|x-y|^k \la y\ra^{-\beta}$ is 
	Hilbert-Schmidt, and hence absolutely bounded, 
	provided $\beta>\frac{n}{2}+k$.  So that
	the operators $vG_j v$ is Hilbert Schmidt for 
	$j\geq \frac{n}{2}-2$ provided
	$\beta>j+2-\frac{n}{2}$.

\end{proof}

\begin{rmk}

	The error estimates here can be more compactly
	summarized as
	\begin{align*}
		M_0^{\pm}(\lambda)=\widetilde O_{\frac{n-1}{2}}
		(\lambda^{n-2+\ell}), \qquad
		M_1^{\pm}(\lambda)=\widetilde O_{\frac{n+1}{2}}
		(\lambda^{n+\ell}), \qquad
		M_2^{\pm}(\lambda)=\widetilde O_{\frac{n+1}{2}}
		(\lambda^{n+2+\ell}),
	\end{align*}
	as absolutely bounded operators on $L^2(\R^n)$,
	for $0<\lambda<\lambda_1$.

\end{rmk}

We note that $U+vG_0v$ is not invertible if there
is an eigenvalue at zero, see 
Lemma~\ref{S characterization} below.
Define $S_1$ to be the Riesz projection onto the 
kernel of $U+vG_0v$ as an operator on $L^2(\R^n)$.  
Then the operator $U+vG_0v+S_1$ is invertible on 
$L^2$, accordingly we define
\begin{align}
	D_0:=(U+vG_0v+S_1)^{-1}.
\end{align}
This operator can be seen to be absolutely bounded.

\begin{lemma}\label{lem:D0bdd}

	If $|V(x)|\les \la x\ra^{-\frac{n+1}{2}-}$, then
	the operator $D_0$ is absolutely bounded in $L^2(\mathbb R^n)$.

\end{lemma}

\begin{proof}

	We  have the resolvent identity
	$$
		D_0=U-D_0(vG_0v+S_1)U.
	$$
	Iterating this identity, we have for $k> 1$,
	\begin{align*}
		D_0&=U+\sum_{j=1}^{k-1} (-1)^jU((vG_0v+S_1)U)^j
		+(-1)^k D_0[(vG_0v+S_1)U]^k.
	\end{align*}
	Using Lemma~\ref{EG:Lem} as in the proof
	Lemma~\ref{lem:iterated} (see Section~\ref{sec:ests}) one can see that each 
	iteration of $vG_0v$ reduces the local singularity
	by two powers.  Thus, we need to iterate the the
	identities at least $k=\lceil
	\frac{n-3}{4}\rceil$ terms to get
	the final iterated operator to be locally $L^2$.
	
	We note that $U$ is clearly absolutely bounded on
	$L^2(\R^n)$. One can note that
	$$
		0=S_1(U+vG_0v) \quad \Rightarrow \quad
		S_1U=-S_1vG_0v \quad \Rightarrow \quad
		S_1=-S_1vG_0w.
	$$
	To see that the mapping properties of $S_1$ are
	at least as good as those of $vG_0v$.
	Alternatively, since $U+vG_0v$ is a compact
	perturbation of the invertible operator $U$, the
	Fredholm alternative guarantees that the projection
	operator $S_1$ is finite rank.
	
	The operator $G_0$ is a scalar multiple
	of the fractional integral operator $I_2$ which
	is a compact operator on $L^{2,\sigma}\to L^{2,-\sigma}$ for $\sigma>1$ by Lemma~2.3 of
	\cite{Jen}.	 Then $vG_0w$ is absolutely bounded on
	$L^2$ provided $v(x)\les \la x\ra^{-1-}$.
	
	So that
	$D_0$ is the sum of absolutely bounded operators,
	$U$, and a Hilbert-Schmidt operators.  Since
	$D_0[(vG_0v)U]^k$ is a bounded operator composed 
	with a
	Hilbert-Schmidt operator, it is Hilbert-Schmidt.

\end{proof}
The above proof is valid whether zero is regular or not.
When zero is regular, $S_1=0$ so many of the terms considered vanish.

We will apply Lemma~\ref{JNlemma} with $A=M^\pm(\lambda)$ and $S=S_1$,
the Riesz projection onto the kernel of
$U+vG_0v$. Thus, we need to show that $M^{\pm}(\lambda)+S_1$
has a bounded inverse in $L^2(\mathbb R^n)$ and
\begin{align}\label{B defn}
  B_{\pm}(\lambda) =S_1-S_1(M^\pm(\lambda)+S_1)^{-1}S_1
\end{align}
has a bounded inverse in $S_1L^2(\mathbb R^n)$.

\begin{lemma}\label{M+S inverse}

	Suppose that zero is not a regular point of the
	spectrum of $H=-\Delta+V$, and let $S_1$ be the
	corresponding Riesz projection on the the zero
	energy eigenspace.  The for sufficiently small
	$\lambda_1>0$, the operators $M^\pm(\lambda)+S_1$ 
	are invertible for all $0<\lambda<\lambda_1$ as
	bounded operators on $L^2(\R^n)$.  Further,
	for any $0\leq \ell\leq 1$, if 
	$\beta>\frac{n}{2}+\ell$ then
	we have the following expansions
	\begin{align*}
		(M^{\pm}(\lambda)+S_1)^{-1}
		&=D_0+\sum_{j=1}^{\frac{n-3}{2}}
		\lambda^{2j}C_{2j}
		\mp i\lambda^{n-2}D_0vG_{n-2}vD_0
		+\widetilde M_0^\pm(\lambda)
	\end{align*}
	where $\widetilde M_0^{\pm}(\lambda)$ satisfies the
	same bounds as $M_0^{\pm}(\lambda)$ and the operators
	$C_{k}$ are absolutely bounded on $L^2$ with
	real-valued kernels.
	Further, if $\beta>\frac{n}{2}+2+\ell$ then
	\begin{align*}
		\widetilde M_0^{\pm}(\lambda)= 
		\lambda^{n-1} C_{n-1}\pm i\lambda^n C_n
		+\widetilde M_1^\pm(\lambda)
	\end{align*}
	where 	$C_n= D_0vG_{n-2}vD_0vG_2vD_0
	+D_0vG_2vD_0vG_{n-2}vD_0-D_0vG_{n}vD_0$, and
	$\widetilde M_1^{\pm}(\lambda)$ satisfies the
	same bounds as $M_1^{\pm}(\lambda)$.
	Further, if $\beta>\frac{n}{2}+4+\ell$ then	
	\begin{align*}
		\widetilde M_1^{\pm}(\lambda)=
		\lambda^{n+1}C_{n+1} \pm i\lambda^{n+2} C_{n+2}+
		\widetilde M_2^\pm (\lambda)
	\end{align*}
	where 
	$\widetilde M_2^{\pm}(\lambda)$ satisfies the
	same bounds as $M_2^{\pm}(\lambda)$.

\end{lemma}

\begin{proof}

We use a Neumann series expansion.  We show the case of
$M^+$ and omit the superscript, the `-' case follows
similarly.  Using \eqref{M0exp} we have
\begin{multline*}
	(M(\lambda)+S_1)^{-1}=(U+vG_0v+S_1
	+\sum_{j=1}^{\frac{n-3}{2}}
	\lambda^{2j}vG_{2j}v+i\lambda^{n-2}vG_{n-2}v
	+M_0(\lambda))^{-1}\\
	=D_0(\mathbbm 1+\sum_{j=1}^{\frac{n-3}{2}}
	\lambda^{2j}vG_{2j}vD_0+i\lambda^{n-2}vG_{n-2}vD_0
	+M_0(\lambda)D)^{-1}\\
	=D_0-\lambda^2 D_0vG_2vD_0+\lambda^4
	[D_0vG_2vD_0vG_2vD_0-D_0vG_4vD_0]
	+\sum_{j=3}^{\frac{n-3}{2}} \lambda^{2j} C_{2j}
	\\-D_0[i\lambda^{n-2}vG_{n-2}v+M_0(\lambda)]D_0
	+\lambda^{2}[D_0vG_{2}vD_0
	[i\lambda^{n-2}vG_{n-2}v+M_0(\lambda)]D_0\\
	+\lambda^2 D_0[i\lambda^{n-2}vG_{n-2}v+M_0(\lambda)]D_0vG_2vD_0]
	+\lambda^{4}[D_0vG_{4}vD_0[i\lambda^{n-2}vG_{n-2}v+M_0(\lambda)]]D_0\\
	+\lambda^4 D_0[i\lambda^{n-2}vG_{n-2}v+M_0(\lambda)]G_4vD_0]+\widetilde M_2(\lambda).
\end{multline*}
One can find explicitly the operators $C_{2j}$ in terms of
$D_0$ and the operators $G_{2k}$, but this is not worth
the effort.  One only needs that these operators are
absolutely bounded with real-valued kernels.  These 
properties are inherited from $D_0$ and $vG_{2k}v$
using that the composition of absolutely bounded operator
are absolutely bounded.

What is important, in our analysis in Section~\ref{sec:disp}, is the odd powers of 
$\lambda$.
We note that the first odd power of $\lambda$ arises in
\begin{multline*}
	-D_0[i\lambda^{n-2}vG_{n-2}v+M_0(\lambda)]D_0
	=-D_0[i\lambda^{n-2}vG_{n-2}v+\lambda^{n-1}
	vG_{n-1}v+i\lambda^n vG_n v+
	M_1(\lambda)]D_0\\
	=-i\lambda^{n-2}D_0vG_{n-2}vD_0
	-\lambda^{n-1}D_0vG_{n-1}
	vD_0-i\lambda^n D_0vG_nvD_0-D_0 M_1(\lambda) D_0
\end{multline*}
The next odd power occurs from the $i\lambda^n D_0vG_nvD_0$ 
term
and the `$x^2$' term in the Neumann series,
that is the term with both $G_2$ and $i\lambda^{n-2}vG_{n-2}v$, accordingly
we see that the $\lambda^n$ term is given by
\begin{align*}
	i\lambda^n[ D_0vG_{n-2}vD_0vG_2vD_0
	+ D_0vG_2vD_0vG_{n-2}vD_0-D_0vG_{n}vD_0]
\end{align*}
The operators that accompany $\lambda^{n+2}$ are not
calculated explicitly, but again are absolutely bounded.
The error bounds follow from the bounds in 
Lemma~\ref{lem:Mexp} and the Neumann series expansion
above.

\end{proof}

\begin{rmk}\label{rmk:reg}

	We note here that is zero is regular the above
	Lemma suffices to establish the dispersive estimates
	using the techniques in Sections~\ref{sec:disp}
	and \ref{sec:finitebs}.
	In this case, $S_1=0$, $D_0=(U+vG_0v)^{-1}$ is still absolutely
	bounded and we have the expansion
	\begin{align*}
		M^{\pm}(\lambda)^{-1}
		&=D_0+\sum_{j=1}^{\frac{n-3}{2}}
		\lambda^{2j}C_{2j}
		\mp i\lambda^{n-2}D_0vG_{n-2}vD_0
		+\widetilde M_0^\pm(\lambda),
	\end{align*}
	with $C_{2j}$ real-valued, absolutely bounded operators.

\end{rmk}

Now we turn to the operators $B_{\pm}(\lambda)$ for
use in Lemma~\ref{JNlemma}.  
Recall that
$$
	B_{\pm}(\lambda)=S_1-S_1(M^\pm(\lambda)+S_1)^{-1}S_1,
$$
and that $S_1D_0=D_0S_1=S_1$.  Thus
\begin{align*}
	B_{\pm}(\lambda)&=S_1-S_1[D_0+
	\sum_{j=1}^{\frac{n-3}{2}}\lambda^{2j}C_{2j}
	\mp i\lambda^{n-2}D_0vG_{n-2}vD_0
	+\widetilde M_0^{\pm}
	(\lambda)]S_1\\
	&=
	-\sum_{j=1}^{\frac{n-3}{2}}
	\lambda^{2j}S_1 C_{2j}S_1
	\pm i\lambda^{n-2}S_1vG_{n-2}vS_1
	-S_1\widetilde M_0^\pm(\lambda)S_1\\
	&=-\lambda^2 S_1vG_2vS_1
	-\sum_{j=2}^{\frac{n-3}{2}}
	\lambda^{2j}S_1 C_{2j}S_1
	\pm i\lambda^{n-2}S_1vG_{n-2}vS_1
	-S_1\widetilde M_0^\pm(\lambda)S_1	
\end{align*}
So that the invertibility of $B_{\pm}(\lambda)$
hinges upon the invertibility of the operator
$S_1D_0vG_2vD_0S_1=S_1vG_2vS_1$ on $S_1L^2(\R^n)$, 
which is
established in Lemma~\ref{D1 lemma} below.  Thus
we define $D_1:=(S_1vG_2vS_1)^{-1}$ as
an operator on $S_1L^2(\R^n)$.  Noting that
$D_1=S_1D_1S_1$, it is clear that $D_1$ is absolutely
bounded.

\begin{lemma}\label{lem:Binv}

	We have the following expansions, if $\beta>\frac{n}{2}+\ell$ for any $0\leq \ell\leq 1$,
	then 
	\begin{align*}
		B_{\pm}(\lambda)^{-1}
		&=-\frac{D_1}{\lambda^2}+\sum_{j=2}^{\frac{n-3}{2}}
		\lambda^{2j-4}B_{2j}\pm i\lambda^{n-6}D_1 vG_{n-2}vD_1
		+\widetilde B_0^\pm(\lambda),
	\end{align*}
	where $\widetilde B_0^{\pm}(\lambda)$ satisfies the
	same bounds as $\lambda^{-4} M_0^{\pm}(\lambda)$ and the operators
	$B_{k}$ are absolutely bounded on $L^2$ with
	real-valued kernels.
	Further, if $\beta>\frac{n}{2}+2+\ell$ then
	\begin{align*}
		\widetilde B_0^{\pm}(\lambda)=
		\lambda^{n-5} B_{n-1}\pm i\lambda^{n-4} B_n
		+\widetilde B_1^\pm(\lambda),
	\end{align*}
	where
		$B_n=D_1vG_{n-2}vD_0vG_2vD_1
		+D_1vG_2vD_0vG_{n-2}vD_1-D_1vG_{n}vD_1
		-D_1vG_{n-2}vD_1C_4D_1-D_1C_4D_1vG_{n-2}vD_1$, and
	 $\widetilde B_1^{\pm}(\lambda)$ satisfies the
	same bounds as $\lambda^{-4} M_1^{\pm}(\lambda)$.
	Further, if $\beta>\frac{n}{2}+4+\ell$ then	
	\begin{align*}
		\widetilde B_1^{\pm}(\lambda)=\lambda^{n-3}B_{n+1} 
		\pm i\lambda^{n-2}B_{n+2}+
		\widetilde B_2^\pm (\lambda),
	\end{align*}
	where 
	$\widetilde B_2^{\pm}(\lambda)$ satisfies the
	same bounds as $\lambda^{-4} M_2^{\pm}(\lambda)$.

\end{lemma}

\begin{proof}

As usual we consider the `+' case and omit subscripts,
the `-' case follows similarly.  We begin by noting that
\begin{align}
	B^{-1}(\lambda)&=[-\lambda^2 S_1vG_2vS_1
	-\sum_{j=2}^{\frac{n-3}{2}}
	\lambda^{2j}S_1 C_{2j}S_1
	\pm i\lambda^{n-2}S_1vG_{n-2}vS_1
	-S_1\widetilde M_0^\pm(\lambda)S_1]^{-1}\nn\\
	&=-\frac{D_1}{\lambda^2}[\mathbbm 1 +
	\sum_{j=2}^{\frac{n-3}{2}}
	\lambda^{2j-2}S_1 C_{2j}S_1D_1
	\mp i\lambda^{n-4}S_1vG_{n-2}vS_1D_1
	+\lambda^{-2}S_1\widetilde M_0^\pm(\lambda)
	S_1D_1]^{-1}.\label{Binv1}
\end{align}
We again only concern ourselves with explicitly finding
the operators for the first two odd powers of $\lambda$
that occur.  This again follows by a careful analysis
of the various terms that arise in the Neumann series
expansion.  It is clear from the expansion
\eqref{Binv1} that the error terms for 
$B^{-1}(\lambda)$ are of size $\lambda^{-4}\widetilde M_{j}(\lambda)$ by performing the Neumann series expansion.

Since the odd powers of $\lambda$ in this
expansion can only arise from odd powers in the
expansion for $B^\pm(\lambda)$ in combination with an
even power of $\lambda$ in the Neumann series we can
identify these operators explicitly.

We note that the $\lambda^{n-6}$ term in the expansion
can come only from the $\lambda^{n-2}$ term in the
expansion in Lemma~\ref{M+S inverse}.  In the `$x$'
term of the Neumann series we see
$$
	-\frac{D_1}{\lambda^2}  \big(\lambda^{-2}\big)
	\big(
	\mp i\lambda^{n-4}S_1vG_{n-2}vS_1D_1
	\big)=\pm i \lambda^{n-6} D_1vG_{n-2}vD_1
$$
The next odd power of $\lambda$ is the $\lambda^{n-4}$.
This picks up contributions from the `$x$' term in the
Neumann series from the contribution of the $\lambda^n$
power in the expansion for $\widetilde M_0^\pm(\lambda)$ in 
Lemma~\ref{M+S inverse}, 
\begin{align*}
	-\frac{D_1}{\lambda^2}&\big(\mp i \lambda^{n-2}S_1C_nD_1\big)\\
	&=\pm i \lambda^{n-4}
	D_1 [D_0vG_{n-2}vD_0vG_2vD_0
	+D_0vG_2vD_0vG_{n-2}vD_0-D_0vG_{n}vD_0]D_1\\
	&=\pm i\lambda^{n-4}[D_1vG_{n-2}vD_0vG_2vD_1
	+D_1vG_2vD_0vG_{n-2}vD_1-D_1vG_{n}vD_1].
\end{align*}
The other contribution is from the ` $x^2$ term of
the Neumann series, specifically the interaction of the
$\lambda^2S_1C_4D_1$ term in \eqref{Binv1} and the
$\lambda^{n-2}$ term from Lemma~\ref{M+S inverse}.
These contribute
\begin{align*}
	-\frac{D_1}{\lambda^2}&\bigg\{\big(\lambda^2 S_1C_4D_1
	\big)\big(\lambda^{n-4}S_1vG_{n-2}vD_1\big)
	+ \big(\lambda^{n-4}S_1vG_{n-2}vD_1
	\big)\big(\lambda^2 S_1C_4D_1\big)
	\bigg\}\\
	&=-\lambda^{n-4}\big(D_1C_4D_1vG_{n-2}vD_1+D_1vG_{n-2}
	vD_1C_4D_1\big).
\end{align*}

\end{proof}

\begin{rmk}

	The error estimates here can be more compactly
	summarized as
	\begin{align*}
		B_0^{\pm}(\lambda)=\widetilde O_{\frac{n-1}{2}}
		(\lambda^{n-6+\ell}), \qquad
		B_1^{\pm}(\lambda)=\widetilde O_{\frac{n+1}{2}}
		(\lambda^{n-4+\ell}), \qquad
		B_2^{\pm}(\lambda)=\widetilde O_{\frac{n+1}{2}}
		(\lambda^{n-2+\ell}),
	\end{align*}
	as absolutely bounded operators on $L^2(\R^n)$,
	for $0<\lambda<\lambda_1$.

\end{rmk}

Later on, inspired by Remark~8.3 in \cite{Jen}
we consider eigenfunctions with certain
orthogonality conditions.  Accordingly, we consider 
two additional cases: first when 
$P_eVG_{n-2}=c_{n-2} P_eV1=0$, and secondly when $P_eV1=0$ and
$P_{e}Vx=0$ which, we see in the next Corollary, implies that $P_{e}VG_{n}VP_e=0$.

\begin{corollary}

	Under the hypotheses of Lemma~\ref{lem:Binv}, if
	$P_eV1=0$ then the first odd power of 
	$\lambda$ in the expansion occurs at $\lambda^{n-4}$.
	That is, if $\beta>\frac{n}{2}+2+\ell$
	\begin{align*}
		B_{\pm}(\lambda)^{-1}
		&=-\frac{D_1}{\lambda^2}+\sum_{j=2}^{\frac{n-3}{2}}
		\lambda^{2j-4}B_{2j}
		+\lambda^{n-5} B_{n-1}\pm i\lambda^{n-4} B_n
		+\widetilde B_1^\pm(\lambda).
	\end{align*}
	If, in addition, $P_{e}Vx=0$ then
	$B_n=0$ in the expansion for $B_{\pm}(\lambda)^{-1}$.
	That is, if $\beta>\frac{n}{2}+4+\ell$
	\begin{align*}
		B_{\pm}(\lambda)^{-1}
		&=-\frac{D_1}{\lambda^2}+\sum_{j=2}^{\frac{n-3}{2}}
		\lambda^{2j-4}B_{2j}
		+\lambda^{n-5} B_{n-1}
		+\lambda^{n-3}B_{n+1} \pm i\lambda^{n-2}B_{n+2}+
		\widetilde B_2^\pm (\lambda).
	\end{align*}

\end{corollary}

\begin{proof}

	The first claim follows clearly from Lemma~\ref{lem:Binv} since the coefficient of
	$\lambda^{n-6}$ is a scalar multiple of the
	operator $P_eV1$.  This is seen through	
	the identities,
	\begin{align}\label{S1 trick}
		S_1=-S_1vG_0w=-wG_0vS_1.
	\end{align}
	Along with Lemma~\ref{lem:eproj} and the fact that
	$D_1=S_1D_1S_1$ we see that
	$$
		D_1=S_1D_1S_1=wG_0vS_1D_1S_1vG_0w=wP_{e}w
	$$
	with $P_e$ the projection onto the eigenspace at zero.
	Then
	$$
		D_1 vG_{n-2}vD_1=wP_{e}VG_{n-2}VP_{e}w=0.
	$$
	We recall that $G_{n-2}=c_{n-2}1$ in the last step.
	
	The second claim follows from the above observations,
	the form of $B_n$ found in Lemma~\ref{lem:Binv}
	and 
	\begin{align*}
	c_n^{-1}D_1vG_{n}vD_1&=wP_{e}V(|x|^2-2x\cdot y +|y|^2)VP_ew\\
	&=wP_eV|x|^2 1VP_ew-2wP_eVx\cdot yVP_ew
	+wP_eV1|y|^2VP_ew.
	\end{align*}

\end{proof}

We can now state several versions the expansions for
$M^\pm (\lambda)^{-1}$.  These different
expansions allow us to have finer control on the time decay rate of the error terms of the evolution given
in Theorem~\ref{thm:main} at the cost of more decay on
the potential.

\begin{prop}\label{prop:Minvfull}

	If $|V(x)|\les \la x\ra^{-n-8-}$, then
	\begin{align*}
		M^{\pm}(\lambda)^{-1}&=-\frac{D_1}{\lambda^2}
		+\sum_{j=0}^{\frac{n-7}{2}} \lambda^{2j}
		M_{2j}\pm i \lambda^{n-6}D_1vG_{n-2}vD_1
		+\lambda^{n-5}
		M_{n-5}\pm i\lambda^{n-4}M_{n-4}\\
		&+\lambda^{n-3}M_{n-3}
		\pm i\lambda^{n-2}M_{n-2}+
		\widetilde O_{\frac{n+1}{2}}(\lambda^{n-2+})
	\end{align*}
	with the operators
	$M_{k}$ all real-valued and absolutely bounded,
	provided that $\lambda$ is sufficiently small.

\end{prop}

\begin{proof}

This follows from the expansions in Lemmas~\ref{M+S inverse} and \ref{lem:Binv}, in the inversion lemma,
Lemma~\ref{JNlemma}.

\end{proof}

This expansion can be truncated, saving required decay
on the potential.  In particular, this longer expansion
can be used to prove Corollary~\ref{cor:ugly}.
We note the following useful
and immediate corollaries.

\begin{corollary}\label{lem:Minv}

	If $|V(x)|\les \la x\ra^{-n-}$, then
	\begin{align}\label{Minv:nocancweak}
		M^{\pm}(\lambda)^{-1}&=-\frac{D_1}{\lambda^2}
		+\sum_{j=0}^{\frac{n-7}{2}} \lambda^{2j}
		M_{2j}\pm i \lambda^{n-6}D_1vG_{n-2}vD_1
		+\widetilde O_{\frac{n-1}{2}}(\lambda^{n-6+}).
	\end{align}
	If $|V(x)|\les \la x\ra^{-n-4-}$, then
	\begin{align}\label{Minv:nocancstrong}
		M^{\pm}(\lambda)^{-1}&=-\frac{D_1}{\lambda^2}
		+\sum_{j=0}^{\frac{n-7}{2}} \lambda^{2j}
		M_{2j}\pm i \lambda^{n-6}D_1vG_{n-2}vD_1+
		\lambda^{n-5}M_{n-5}
		+\widetilde O_{\frac{n-1}{2}}(\lambda^{n-4}).
	\end{align}	

\end{corollary}

\begin{proof}

The first claim
follows from the expansions in Lemmas~\ref{M+S inverse} and \ref{lem:Binv}, in the inversion lemma,
Lemma~\ref{JNlemma}, going out to the error terms
$M_0^{\pm}(\lambda)$ and $B_0^{\pm}(\lambda)$ respectively
using $\ell=0+$.

The second claim follows from going further in the
expansions in Lemmas~\ref{M+S inverse} and
\ref{lem:Binv}, reaching $M_1^{\pm}(\lambda)$ and $B_1^{\pm}(\lambda)$  respectively and taking $\ell=0$.

\end{proof}

\begin{corollary}\label{cor:Minv}

	If $P_eV1=0$ and
	$|V(x)|\les \la x\ra^{-n-4-}$, then
	\begin{align}\label{Minv PV1=0}
		M^{\pm}(\lambda)^{-1}&=-\frac{D_1}{\lambda^2}
		+\sum_{j=0}^{\frac{n-7}{2}} \lambda^{2j}
		M_{2j}
		+\lambda^{n-5}
		M_{n-5}\pm i\lambda^{n-4}M_{n-4}
		+\widetilde O_{\frac{n-1}{2}}(\lambda^{n-4+}). 
	\end{align}
	If $|V(x)|\les \la x \ra^{-n-8-}$, then
	\begin{align}\label{Minv PV1strong}
			M^{\pm}(\lambda)^{-1}&=-\frac{D_1}{\lambda^2}
			+\sum_{j=0}^{\frac{n-7}{2}} \lambda^{2j}
			M_{2j}+\lambda^{n-5}M_{n-5}\pm i\lambda^{n-4}M_{n-4}
			+\lambda^{n-3}M_{n-3}\\
			&\pm i\lambda^{n-2}M_{n-2}+
			\widetilde O_{\frac{n+1}{2}}(\lambda^{n-2+}).\nn
	\end{align}	
	If in addition, $P_eVx=0$, and
	$|V(x)|\les \la x\ra^{-n-8-}$, then
	\begin{align}
		M^{\pm}(\lambda)^{-1}&=-\frac{D_1}{\lambda^2}
		+\sum_{j=0}^{\frac{n-7}{2}} \lambda^{2j}
		M_{2j}+\lambda^{n-5}M_{n-5}
		+\lambda^{n-3}M_{n-3}\\
		&\pm i\lambda^{n-2}M_{n-2}+
		\widetilde O_{\frac{n+1}{2}}(\lambda^{n-2+}).\nn
	\end{align}

\end{corollary}

\section{Dispersive estimates: the leading terms}\label{sec:disp}

In this section we use the expansions developed for
$M^{\pm}(\lambda)^{-1}$ in Section~\ref{sec:resolv} to
develop small $\lambda$ expansions for the operators
$R_V^\pm(\lambda^2)$ to determine the time decay rate
of the evolution $e^{itH}P_{ac}(H)$.
We divide our analysis into three cases that depend
on the eigenfunction cancellation.  Specifically, we
need to consider when $P_eV1\neq 0$, then when
$P_eV1=0$ but $P_eVx\neq 0$, and finally we
consider when both $P_eV1=0$ and $P_eVx=0$.

Iterating the standard resolvent identity
$$
	R_V^\pm(\lambda^2)=R_0^\pm(\lambda^2)
	-R_0^\pm(\lambda^2)VR_V^\pm(\lambda^2)
	=R_0^\pm(\lambda^2)
	-R_V^\pm(\lambda^2)VR_0^\pm(\lambda^2)
$$
we have the  identity
\begin{align}
	R_V^\pm(\lambda^2)&=\sum_{k=0}^{2m+1}(-1)^k
	R_0^\pm(\lambda^2)[VR_0^\pm(\lambda^2)]^k\label{eq:finitebs}\\
	&+[R_0^\pm(\lambda^2)V]^{m}R_0^\pm(\lambda^2)v
	M^{\pm}(\lambda)^{-1} 
	vR_0^\pm(\lambda^2)[VR_0^\pm(\lambda^2)]^{m}.\label{eq:bstail}
\end{align}
In light of Lemma~\ref{lem:iterated} the identity 
holds for $m+1\geq \frac{n-3}{4}$ and 
$|V(x)|\les \la x\ra^{-\frac{n+1}{2}-}$ as an identity
from $L^{2,\f12+}\to L^{2,-\f12-}$, as in the 
limiting absorption principle.

The most singular $\lambda$ terms of the expansion,
and hence the slowest decaying in time terms occur in the
last term involving the operator $M^{\pm}(\lambda)^{-1}$.
These are the terms whose time decay rate is sensitive to
the existence of zero-energy eigenvalues.
In this section we prove the dispersive bounds for
tail of the Born series.
Accordingly, we focus on the contribution of
\eqref{eq:bstail} and put the 
dispersive bounds for the finite Born series terms,
\eqref{eq:finitebs}, in Section~\ref{sec:finitebs}.
In particular, we note that Proposition~\ref{bsprop} ensures 
these terms contribute $|t|^{-\f n2}$ 
as an operator from $L^1$ to $L^\infty$ to the evolution
of $e^{itH}P_{ac}(H)$.

From the `+/-' cancellation, we need only concern
ourselves with the odd powers of $\lambda$ that 
arise in the expansion of
\begin{multline}\label{eq:nastydiff}
	(R_0^+(\lambda^2)V)^mR_0^+(\lambda^2) v
	M^+(\lambda)^{-1}v R_0^+(\lambda^2) (V
	R_0^+(\lambda^2))^m\\
	-(R_0^-(\lambda^2)V)^mR_0^-(\lambda^2) v
	M^-(\lambda)^{-1}v R_0^-(\lambda^2) (V
	R_0^-(\lambda^2))^m.
\end{multline}
Using the algebraic fact,
\begin{align}\label{alg fact}
	\prod_{k=0}^MA_k^+-\prod_{k=0}^M A_k^-
	=\sum_{\ell=0}^M \bigg(\prod_{k=0}^{\ell-1}A_k^-\bigg)
	\big(A_\ell^+-A_\ell^-\big)\bigg(
	\prod_{k=\ell+1}^M A_k^+\bigg),
\end{align}
We have two cases to consider, either the `+/-' 
difference in \eqref{alg fact} acts on a resolvent or on the operators
$M^{\pm}(\lambda)^{-1}$.

\subsection{No cancellation}
We first consider the case in which there is no
eigenfunction cancellation, that is when the operator
$P_eV1\neq 0$.  In this case,

\begin{lemma}\label{lem:nocanc}

	If $P_eV1\neq 0$ and $|V(x)|\les \la x \ra^{-n-4-}$, 
	then
	$$
		\eqref{eq:nastydiff}=c_{n-2}\lambda^{n-6} 
		P_eV1VP_e+\frac{E_0^+(\lambda)-E_0^-(\lambda)}
		{\lambda^2}VP_e+
		P_eV\frac{E_0^+(\lambda)-E_0^-(\lambda)}
				{\lambda^2}
		+\widetilde O_{\frac{n-1}{2}} (\lambda^{n-4}),
	$$
	which contributes  $C_n|t|^{2-\f n2}P_eV1VP_e +
	O(|t|^{1-\f n2})$ to \eqref{Stone}.

\end{lemma}

\begin{proof}

We first consider when the `+/-' difference acts on
free resolvent operators $R_0^\pm$.
For the other resolvents, by Lemma~\ref{lem:R0exp}, we can write
$$
	R_0^\pm(\lambda^2)=G_0+\widetilde O_{\frac{n-1}{2}}(\lambda^2)
$$
and by Corollary~\ref{lem:Minv},
$$
	M^{\pm}(\lambda)^{-1}=-\frac{D_1}{\lambda^2}+
	\widetilde O_{\frac{n-1}{2}}(1).
$$
Recall that the error terms here are to indicate that
differentiation in $\lambda$ is comparable to division
by $\lambda$, up to order $\frac{n-1}{2}$.  

For the purposes of making sure that the $O(|t|^{1-\f n2})$ remainder maps between
unweighted spaces, the most delicate
term will be of the form
\begin{align}\label{eq:delicate}
	[R_0^+(\lambda^2)-R_0^-(\lambda^2)]\bigg\{(VG_0)^m v
	\frac{D_1}{\lambda^2}v (G_0V)^mG_0+\widetilde O_{\frac{n-1}{2}}(1)\bigg\}
\end{align}
This occurs when the `+/-' difference in \eqref{eq:nastydiff} affects the leading free resolvents
and using the expansions above for the remaining resolvents
and $M^{\pm}(\lambda)^{-1}$.

Then using the identity $S_1=-S_1vG_0w=-wG_0vS_1$,
we have $G_0VG_0vD_1=G_0vwG_0vS_1D_1=-G_0vD_1$
and similarly $D_1vG_0VG_0=-D_1vG_0$.  
\begin{align*}
	V(G_0V)^{m-1}G_0 vD_1vG_0(VG_0)^m
	=-VG_0vD_1vG_0
	=-VP_e.
\end{align*}
Where we used
the definition of $P_e$, \eqref{Pe defn}, in the
last step.  Now, using
the expansion for $R_0^\pm(\lambda^2)$ in
Lemma~\ref{lem:R0exp} with $\ell=0$, we need only bound 
the contribution of
\begin{align*}
	\frac{E_0^\pm(\lambda)}{\lambda^2}VP_e+\widetilde O_{\frac{n-1}{2}}(\lambda^{n-2}).
\end{align*}
By Lemma~\ref{lem:R0exp} with $\ell=0$, $E_0^\pm (\lambda)=\widetilde O_{\frac{n-1}{2}}(\lambda^{n-2})$, thus
the first term is of size $\widetilde O_{\frac{n-1}{2}}
(\lambda^{n-4})$, which
by Lemma~\ref{lem:fauxIBP}, 
we can bound the contribution
of the above term to \eqref{Stone} by
$|t|^{1-\frac{n}{2}}$.  
The error term is an operator from $L^1\to L^\infty$ since we can bound it by a  constant independent of $x$ and $y$, with the intermediate $z_j$ integrals are controlled
by Lemma~\ref{lem:iterated} and the fact that the terms
in the expansion of $M^\pm(\lambda)^{-1}$ are absolutely bounded on $L^2$.  The first term can similarly be seen to map $L^1$ to $L^\infty$ since the decay of $V$ and the
bounds on $E_0^\pm$ in Lemma~\ref{lem:R0exp} ensure
that $VE_0^\pm(\lambda)$ is in $L^1$, then
$P_e:L^1\to L^\infty$ by Corollary~\ref{Pemapping}.
There is, of course, a similar
term where the difference of resolvents occurs at the
lagging resolvent, this contributes terms of the 
form
\begin{align*}
	P_eV\frac{E_0^\pm(\lambda)}{\lambda^2}+\widetilde O_{\frac{n-1}{2}}(\lambda^{n-2}),
\end{align*}
which are controlled similarly.

The other terms in which the difference of resolvents
occurs with `inner resolvents' are similar. 
In fact, we note that if we have the difference on an
`inner resolvent' as in
\begin{align*}
	(R_0^-V)^j&[R_0^+(\lambda^2)-R_0^-(\lambda^2)]
	(VR_0^+)^{m-j} v
	\frac{D_1}{\lambda^2}v (R_0^+V)^mR_0^+\\
	&=(G_0V)^j[2i \lambda^{n-4}G_{n-2} +E_1^+(\lambda)
	-E_1^-(\lambda)]
	(VG_0)^{m-j} vD_1v (G_0V)^mG_0+\widetilde O_{\frac{n-1}{2}}(\lambda^{n-2}),
\end{align*}
where we used Lemma~\ref{lem:R0exp}.  From the bounds on
the error terms $E_1^{\pm}(\lambda)$
we can integrate by parts safely as the $|z_j-z_{j+1}|$ powers
can be absorbed by the decay of the potentials that
are on either side.

The most singular term with respect to $\lambda$ occurs from the first
odd power of $\lambda$  that appears in the expansion
of $M^{\pm}(\lambda)$.  In this case, 
we need to bound
\begin{align*}
	(R_0^-(\lambda^2)V)^mR_0^-(\lambda^2) v[M^+(\lambda)^{-1}-M^-(\lambda)^{-1}]
	vR_0^+(\lambda^2)(VR_0^+(\lambda^2))^m.
\end{align*}
By Lemma~\ref{lem:R0exp} and Corollary~\ref{lem:Minv}
this contributes
$$
	\lambda^{n-6} (G_0V)^mG_0 vD_1vG_{n-2}vD_1		
	vG_0(VG_0)^m+\widetilde O_{\frac{n-1}{2}}(\lambda^{n-4}).
$$
The remaining terms, which we put in the error term, are all of size $\lambda^{n-4}$ or
smaller by \eqref{Minv:nocancstrong} and Lemma~\ref{lem:R0exp}.  
The operator with $\lambda^{n-6}$ can be seen to be a
multiple of $P_eV1VP_e$ similar to the analysis 
for \eqref{eq:delicate}, we have
\begin{multline*}
	(G_0V)^mG_0 vD_1vG_{n-2}vD_1vG_0(VG_0)^m
	=G_0vD_1vG_{n-2}vD_1vG_0\\
	=G_0vD_1v G_0wv G_{n-2}
	vwG_0vD_1vG_0 = P_eVG_{n-2}VP_e=c_{n-2}P_eV1VP_e.
\end{multline*}
Where we used \eqref{Pe defn} in the second to last
equality.
By Lemma~\ref{lem:IBP} we can bound the 
contribution of this to the Stone
formula \eqref{Stone} with
$C_n|t|^{2-\frac{n}{2}}P_eV1VP_e+O(|t|^{1-\f n2})$.  

\end{proof}

The remaining terms in the Born series, i.e. those derived from~\eqref{eq:finitebs},
have more rapid decay for large $|t|$ by Proposition~\ref{bsprop}.
In fact using the identities for $S_1$ and
Lemma~\ref{lem:eproj}, at this point we can write
\begin{align}
	e^{itH}P_{ac}(H)=C_n|t|^{2-\f n2} P_eV
	1VP_e+O(|t|^{1-\f n2}),
\end{align}
where the operator $P_eV1VP_e$ is rank one.  The weaker
claim, with error term of size $o(|t|^{2-\f n2})$ follows
by using \eqref{Minv:nocancweak} in place of
\eqref{Minv:nocancstrong} to obtain error terms of size
$\widetilde O_{\frac{n-1}{2}}(\lambda^{n-6+})$.

\subsection{The case of $P_eV1=0$}

	In Theorem~\ref{thm:main}, we have that if the
	operator $P_eV1=0$, one can achieve faster time
	decay.  In particular,

\begin{lemma}\label{lem:1canc}

	If $P_eV1= 0$ and $|V(x)|\les \la x \ra^{-n-8-}$, 
	then
	$$
		\eqref{eq:nastydiff}=\lambda^{n-4} 
		\Gamma_1+\frac{E_1^+(\lambda)-E_1^-(\lambda)}
		{\lambda^2}VP_e+
		P_eV\frac{E_1^+(\lambda)-E_1^-(\lambda)}
				{\lambda^2}
		+\la x \ra \la y \ra 
		\widetilde O_{\frac{n+1}{2}} (\lambda^{n-2}),
	$$ 
	with $\Gamma_1:L^1\to L^\infty$.
	This contributes  $|t|^{1-\f n2} +
	\la x \ra \la y \ra O(|t|^{-\f n2})$ to \eqref{Stone}.

\end{lemma}

We note that the terms with $\lambda^{-2}E_1^{\pm}(\lambda)$
generically demand polynomial weights of power
two in the spatial
variables $x$ and $y$.  These weights appear, for example, in the
properties of the operator $A_0(t)$ in~\eqref{eq:corollary}.
We prove the more elementary generic bounds here and postpone
a full account of these terms until
Lemmas~\ref{lem:Tlow} and \ref{lem:Thigh} below.
In a similar vein, the error term contribution is clearly bounded by
$\la x\ra \la y\ra|t|^{-\f n2}$ via Lemma~\ref{lem:fauxIBP}.
However Lemma~\ref{lem:BS} will demonstrate how to
bound terms of this type by $|t|^{-\f n2}$ without weights,
by shifting the stationary phase point if all $\frac{n+1}{2}$
derivatives fall on the leading (or lagging) resolvent.
The sharper error estimate is not crucial here but it will be needed
momentarily.

\begin{proof}

In this case, the `+/-' cancellation on
$M^{\pm}(\lambda)^{-1}$ yields a first term of size
$\lambda^{n-4}$.  Specifically, from Corollary~\ref{cor:Minv} using \eqref{Minv PV1strong},
we have
$$
	M^+(\lambda)^{-1}-M^-(\lambda)^{-1}
	=2i\lambda^{n-4}M_{n-4}+2i \lambda^{n-2}M_{n-2}
	+\widetilde O_{\frac{n+1}{2}}(\lambda^{n-2+}) 
$$
As in the previous case, 
we can use
Lemma~\ref{lem:R0exp} to write 
$R_0^{\pm}(\lambda^2)=G_0+\lambda^2 G_2+|x-\cdot|
\widetilde O_{\frac{n+1}{2}} (\lambda^{4})$, and the most singular
expression in $\lambda$ occurs when the free resolvent is approximated by $G_0$
in each instance.
Thus we have
\begin{align*}
	(R_0^-(\lambda^2)V)^m&R_0^-(\lambda^2)v
	[M^+(\lambda)^{-1}-M^-(\lambda)^{-1}]
	vR_0^+(\lambda^2)
	(VR_0^-(\lambda^2)) ^m\\
	&=\lambda^{n-4}(G_0V)^mG_0vM_{n-4}vG_0(VG_0)^m
	+\lambda^{n-2}\Gamma_{n-2}
	+\la x \ra \la y \ra 
	\widetilde O_{\frac{n+1}{2}}(\lambda^{n-2+})
\end{align*}
Where $\Gamma_{n-2}$ is the contribution of
the operator $M_{n-2}$ with all $G_0$s and the
contribution of $M_{n-4}$ with $G_0$s and 
exactly one $G_2$.
This contributes $|t|^{1-\f n2}+\la x \ra \la y \ra 
O(|t|^{-\f n2})$ to the
Stone formula, \eqref{Stone}, by Lemmas~\ref{lem:IBP} and~\ref{lem:fauxIBP} respectively.

On the other hand, if the `+/-' cancellation affects 
an outer resolvent
we have to consider the contribution
of terms of the form
\begin{align*}
	[R_0^+(\lambda^2)-&R_0^-(\lambda^2)](VR_0^+(\lambda^2))^m v
	\frac{D_1}{\lambda^2}v (R_0^+(\lambda^2)V)^mR_0^+(\lambda^2)\\
	&=(2i\lambda^{n-2}G_{n-2}+E_1^+(\lambda)
	-E_1^-(\lambda))(VG_0)^m v
	\frac{D_1}{\lambda^2}v (G_0V)^mG_0
	+\la x\ra \la y\ra \widetilde O_{\frac{n+1}{2}}(\lambda^{n-2}).
\end{align*}
We note that for the remaining terms, we use
Corollary~\ref{cor:Minv} to write
$M^{\pm}(\lambda)^{-1}=-D_1/\lambda^2+ M_0+\widetilde{O}_{\frac{n+1}{2}}(\lambda^2)$.  Noting that
$\widetilde O_k(\lambda^j)=\widetilde O_k(\lambda^{\ell})$
for $\ell\leq j$, 
we can bound all of the terms as we do the first
most singular one.
The first term with the operator $G_{n-2}$ vanishes
since $P_eV1=0$.  Thus, we are left to bound
$$
	\frac{E_1^{\pm}(\lambda)}{\lambda^2} VP_e.
$$
Using the bounds in Lemma~\ref{lem:R0exp}, this is 
of size $|x-\cdot|^{1+\ell} \widetilde O_{\frac{n+1}{2}}(\lambda^{n-3+\ell})$.  Taking $\ell=1$
in Lemma~\ref{lem:fauxIBP} shows that
the contribution of this term is of size
$\la x\ra^2 |t|^{-\f n2 }$ for large $|t|$.
As in Lemma~\ref{lem:nocanc}, terms where the +/- difference acts on an
inner resolvent it can be bounded similarly.  Moreover the factor of $|z_j- z_{j+1}|^{2}$ in
the bound for $E_1^\pm(\lambda)$ is then negated by multiplying by the potential on both sides.

There are only a finite number of
terms in the expansion of finite rank operators
that contribute $\lambda^{n-4}$ to the expansion.
By Lemma~\ref{lem:IBP}, 
these contribute
$|t|^{1-\f n2}$ to the Stone formula.  

\end{proof}

The remaining claims in the second part of
Theorem~\ref{thm:main} come from truncating the
expansion for $R_0^\pm(\lambda^2)$ at $E_0^\pm(\lambda)$
with $\ell=0$ and $\ell=0+$ respectively.

Hence we have if $P_eV1=0$
\begin{align*}
	e^{itH}P_{ac}(H)=|t|^{1-\f n2}\Gamma+O(|t|^{-\f n2})
\end{align*}
where $\Gamma$ is a finite rank operator mapping
$L^1$ to $L^\infty$, which we
do not make explicit and the error term is understood
as an operator between weighted spaces.  Combining this
with the analysis for when $P_eV1\neq 0$, we have the
expansion
\begin{align*}
	e^{itH}P_{ac}(H)=C_n|t|^{2-\f n2}P_eV1VP_e+
	|t|^{1-\f n2}\Gamma+O(|t|^{-\f n2}),
\end{align*}
which is valid whether or not $P_eV1=0$.

\subsection{The case of $P_eV1=0$ and
$P_eVx=0$}

Finally we consider the case with both cancellation
conditions on the eigenfunctions at zero energy.  In
particular, we show that

\begin{lemma}\label{lem:2canc}

	If $P_eV1= 0$, $P_eVx=0$ 
	and $|V(x)|\les \la x \ra^{-n-8-}$, 
	then
	$$
		\eqref{eq:nastydiff}=\frac{R_0^+(\lambda)-R_0^-(\lambda)}
		{\lambda^2}VP_e
		+P_eV\frac{R_0^+(\lambda)-R_0^-(\lambda)}
				{\lambda^2}
		+\mathcal E(\lambda),
	$$ 
	which contributes  $|t|^{-\f n2}$ to \eqref{Stone}.

\end{lemma}

It would be convenient if we could write
$\mathcal E(\lambda)=\widetilde O_{\frac{n+1}{2}} (\lambda^{n-2})$, and this is nearly true.  The issue is that
one cannot  bound  
$(\partial_\lambda)^{\frac{n+1}{2}} E_0^\pm(\lambda)$ as in Lemma~\ref{lem:R0exp}
without introducing spatial weights for large values of $\lambda|x-y|$.
So long as that is the only obstruction, such terms
can still be bounded in unweighted spaces by following the model of the finite Born series, see Lemma~\ref{lem:BS}.
We claim that every term of order $O(\lambda^{n-2})$ in the expansion of~\eqref{eq:nastydiff}
can be handled in this manner, except the ones singled out for special consideration above.

\begin{proof}

The  most delicate case, for the purposes of 
obtaining an unweighted bound, occurs when the
`+/-' difference
acts on a leading or lagging free resolvent.
The goal now is to come up with an unweighted dispersive bound for the expression

\begin{align}\label{R0Pe}
\int_0^1 e^{it\lambda^2} \chi(\lambda) \lambda^{-1}
(R_0^+(\lambda^2) - R_0^-(\lambda^2))(x,z_1)V(z_1)P_e
(z_1,y) \,d\lambda,
\end{align}
where $P_e$ is the finite rank projection onto the nullspace of $-\Delta + V$.  In Lemma~\ref{lem:1canc},
we showed that this term can be bounded by
$ \la x \ra^2 |t|^{-\f n2}$, here we wish to
remove the dependence on $x$, and respectively the
dependence on $y$ for the corresponding term involving
$P_e(x,z_1)V(z_1)(R_0^+(\lambda^2) - R_0^-(\lambda^2))(z_1,y)$.  For clarity, we are
quite explicit about the spatial variable dependence of
the various operators.  In particular, we show that
the dependence on $x$ or $y$ in the result bounds can
be changed into growth in the inner spatial variable
$z_1$, which can be controlled by the decay of the
potential $V(z_1)$.

In odd dimensions, one can write out 
\begin{equation}
R_0^+(\lambda^2)(x,z_1) = \sum_{j=0}^{\frac{n-3}{2}}
\lambda^{2j} G_{2j} + i\lambda^{n-2}G_{n-2} + \lambda^{n-1}G_{n-1} + 
\lambda^{n-2}K_+(\lambda|x-z_1|),
\end{equation} 
see Lemma~\ref{lem:R0exp} above.
As usual $R_0^-(\lambda^2)$ is the complex conjugate of $R_0^+(\lambda^2)$.

Then $R_0^+(\lambda^2) - R_0^-(\lambda^2) =
2i \lambda^{n-2}G_{n-2} + \lambda^{n-2}(K_+(\lambda|x-z_1|)- K_-(\lambda|x-z_1|))$
with the lower-order terms all cancelling each other.
We can already write out
the contribution
\begin{equation} 
\Bigl( \int_0^1 e^{it\lambda^2} \lambda^{n-3} \chi(\lambda) \, d\lambda\Bigr) G_{n-2}VP_e
= (C |t|^{1 - \frac{n}{2}} + O(|t|^{-\frac{n}{2}})) 1VP_e.
\end{equation}
These terms don't require weights, as $1VP_e$ is a bounded operator from $L^1$ to $L^\infty$, this
follows from the mapping properties of $P_e$ in
Corollary~\ref{Pemapping} and the decay of $V$.
And of course they vanish in the case where $1 V P_e = 0$.

Estimates for the terms with $K_\pm(\lambda|x-z_1|)$ do not require additional
cancellation, so we treat only the $K_+$ case.
Near zero, i.e. when $\lambda|x-z_1|\leq \f12$, the 
function $K_+(\lambda|x-z_1|)$ can be viewed as a Taylor series remainder, and for
large values of $\lambda|x-z_1|$ it follows the decay of the resolvent kernel, offset by the
collection of monomial terms  $\lambda^{j+2-n} G_{j}$.   Recall that each
operator $G_j$ has kernel $G_j(x,z_1) = c_j|x-z_1|^{2+j-n}$.  Using the expansions set forth
in Lemma~\ref{lem:R0exp}, one sees that
\begin{equation} \label{eq:K_estimates}
K_+(z) \sim \begin{cases} C_n z^2 &\text{ if } z \leq 1 \\
C_{n-1}z + C_{n-2} + \cdots \frac{C_0}{z^{n-2}} + C\frac{e^{iz}}{z^{\frac{n-1}{2}}} &\text{ if } z > 1
\end{cases} .
\end{equation}

And its derivatives are bounded by
\begin{equation} \label{eq:K'_estimates}
|K_+^{(\ell)}(z)| \les \begin{cases} z^{2-\ell} &\text{ if } z \leq 1 \\
\max(z^{1-\ell}, z^{\frac{1-n}{2}}) &\text{ if } z > 1
\end{cases} .
\end{equation}

This yields the useful uniform bound 
\begin{equation} \label{eq:K_uniform}
|K_+^{(\ell)}(z)| \les z^{2-\ell}, \quad \text{for }0 \leq \ell \leq \frac{n+3}{2},
\end{equation}
which is sufficient for most of the calculations to follow.  More care is required to control
$K_+^{(\ell)}(z)$ in a few terms that we need to differentiate to order $\ell = \frac{n+5}{2}$.

We wish to control the size of
\begin{equation}\label{Kint}
\int_0^{1} e^{it\lambda^2} \chi(\lambda) \lambda^{n-3}K_+(\lambda|x-z_1|)\,d\lambda.
\end{equation}
We break this into to pieces, $\eqref{Kint}:=T_{low}+T_{med}$, which we bound in
Lemmas~\ref{lem:Tlow} and \ref{lem:Thigh} respectively.

If the
difference of free resolvents occurs affects `inner
resolvents', 
$$
	(R_0^-(\lambda^2)V)^j [R_0^+(\lambda^2)-R_0^-(\lambda^2)](VR_0^+(\lambda^2))^{m-j}v \bigg[\frac{D_1}{\lambda^2}+ \widetilde{O}(1)\bigg] (R_0^+V)^m R_0^+(\lambda^2)
$$
the proof is somehow less delicate.  The weights only occur if
$\frac{n+1}{2}$ derivatives act on either the leading
or lagging free resolvent.  Bounding the contribution
of these terms is essentially identical in form to how
one controls the Born series terms in Lemma~\ref{lem:BS}.
To avoid repetition, we refer the reader to the proof
of Lemma~\ref{lem:BS} and give only a brief sketch of the
bound for these terms. 

Using
$R_0^+(\lambda^2)-R_0^-(\lambda^2)= \widetilde O_{\frac{n-1}{2}}(\lambda^{n-2})$, we can control the
$\lambda^{-2}$ singularity.  Then one uses
Lemmas~\ref{recurrence} to see that
$$
	\lambda \bigg(\frac{d}{d\lambda}\frac{1}{\lambda}
	\bigg)^{\frac{n-1}{2}} R_0^\pm (\lambda^2)(x,z_1)
	=e^{\pm i\lambda |x-z_1|}.
$$
At this point, integrating by parts against the imaginary
Gaussian would result in spatial weights.  So 
one uses a modification of stationary phase in Lemma~\ref{lem:statphase}.

Similarly, the next term in the `+/-' difference of the
operators $M^{\pm}(\lambda)^{-1}$ is of order
$\lambda^{n-2}$.  That is, from
Corollary~\ref{cor:Minv} we see
$$
	M^+(\lambda)^{-1}-M^{-}(\lambda)^{-1}=	
	2i\lambda^{n-2}M_{n-2}+\widetilde O_{\frac{n+1}{2}}
	(\lambda^{n-2+}).
$$
As in the previous bounds, the contribution of this
can be bounded  
by $|t|^{-\f n2}$.

\end{proof}

We now consider estimates for the low-energy contribution
\begin{equation*}
T_{low} = \int_0^{|t|^{-\frac12}} e^{it\lambda^2}\chi(\lambda) \lambda^{n-3}K_+(\lambda|x-z_1|)\,d\lambda.
\end{equation*}

\begin{lemma}\label{lem:Tlow}

$T_{low}VP_e$ is a bounded operator from $L^1$ to $L^{\infty, -2}$ whose norm is dominated
by $|t|^{-\frac{n}{2}}$.  If $1VP_e = 0$, the same is true with estimates in $L^{\infty, -1}$.
If both $1VP_e = 0$ and $x_jVP_e = 0$ for each $j \in \{1\ldots n\}$,
then the operator bound is valid between $L^1$ and $L^\infty$.

\end{lemma}

\begin{proof}

First note that $P_e$ is a bounded finite rank operator from $L^1$ to $L^\infty$ by Corollary~\ref{Pemapping}.
Enough decay of the potential is assumed so that $\text{Range}\, VP_e \subset L^{1, 2}$, which follows if $|V(x)| \les \la x\ra^{-n-2-}$.

It follows immediately from~\eqref{eq:K_uniform} that
$|T_{low}(x,z_1)| \les |t|^{-\frac{n}{2}} |x-z_1|^2 \les |t|^{-\frac{n}{2}}\la x\ra^2\la z_1\ra^2$,
 making it a bounded operator from
$L^{1,2}$ to $L^{\infty,-2}$ with norm $|t|^{-n/2}$.

In the event that $1 V P_e = 0$, there is room for improvement due to
the extra cancellation
\begin{equation*}
T_{low}VP_e = \Bigl(\int_0^{|t|^{-\frac12}} e^{it\lambda^2} 
\lambda^{n-3}(K_+(\lambda|x-z_1|) - K_+(\lambda|x|)) \,d\lambda\Bigr) VP_e.
\end{equation*}
The bound on $K_+'$ from~\eqref{eq:K_uniform} and the Mean Value Theorem
imply that
\begin{equation*}
|K_+(\lambda|x-z_1|) - K_+(\lambda|x|)| \les \lambda^2 |z_1| \max(|x|, |x-z_1|)
\les \lambda^2 \la x\ra \la z_1\ra^2,
\end{equation*}
which gives a bounded map from $L^{1,2}$ to $L^{\infty, -1}$ for each $\lambda > 0$. 
The operator norm of $T_{low}VP_e$ is still controlled
by $\int_0^{|t|^{-1/2}} \lambda^{n-1}\,d\lambda \sim |t|^{-\frac{n}{2}}$ as above.

If we further assume that the range of $VP_e$ is orthogonal to all linear functions,
i.e. if $xVP_e=0$, then
another layer of corrections is possible.
\begin{equation*}
T_{low}VP_e = \Bigl(\int_0^{|t|^{-\frac12}} e^{it\lambda^2} 
\lambda^{n-3}\big(K_+(\lambda|x-z_1|) - K_+(\lambda|x|) + 
\lambda K_+'(\lambda|x|) {\textstyle \frac{x}{|x|}} \cdot z_1\big) \,d\lambda\Bigr) VP_e.
\end{equation*}
The Taylor remainder theorem, with respect to $x \in \R^n$, gives an upper bound 
\begin{align*}
\big|K_+(\lambda|x-z_1|) - K_+(\lambda|x|) + \lambda K_+'(\lambda|x|){\textstyle \frac{x}{|x|}} \cdot z_1\big| 
&\les |z_1|^2 \max_z\Bigl(\lambda\frac{|K_+'(\lambda z)|}{|z|} + \lambda^2|K_+''(\lambda z)|\Bigr) \\
&\les \lambda^2\la z_1\ra^2,
\end{align*}
with the last inequality following from~\eqref{eq:K_uniform}.
Using the previous arguments, the range of $T_{low}VP_e$ is now bounded by 
$|t|^{-\f n2}$ times a
constant function in $x$, without polynomial weights. 

\end{proof}

The order of the operators $T_{low}VP_e$ occurs when
the difference of free resolvents from \eqref{alg fact}
occurs on the leading resolvents.  If the 
`+/-' difference
occurs on the lagging resolvents, one must consider
the contribution of $P_eVT_{low}$.

\begin{corollary}

$P_eVT_{low}$ is a bounded operator from $L^{1,2}$ to $L^{\infty}$ whose norm is dominated
by $|t|^{-\frac{n}{2}}$.  If $P_eV1 = 0$, the same is true with estimates from $L^{1, 1}$.
If both $P_eV1 = 0$ and $P_eVx_j = 0$ for each $j \in \{1\ldots n\}$,
then the operator bound is valid between $L^1$ and $L^\infty$.

\end{corollary}

The same operator estimates also hold for the intermediate-energy contribution
\begin{equation*}
T_{mid} = \int_{|t|^{-\frac12}}^1 e^{it\lambda^2} \lambda^{n-3}K_+(\lambda|x-z_1|) \chi(\lambda)\,d\lambda.
\end{equation*}

\begin{lemma}\label{lem:Thigh}

$T_{mid}VP_e$ is  a bounded operator from $L^1$ to $L^{\infty, -2}$ whose norm is dominated
by $|t|^{-\frac{n}{2}}$.  If $1VP_e = 0$, the same is true with estimates in $L^{\infty, -1}$.
If both $1VP_e = 0$ and $x_jVP_e = 0$ for each $j \in \{1\ldots n\}$,
then the operator bound is valid between $L^1$ and $L^\infty$.

\end{lemma}

\begin{proof}
Without loss of generality, we take $t>0$.  The case of
$t<0$ follows similarly with obvious modifications.
Integrate by parts $\frac{n-1}{2}$ times to obtain
\begin{multline} \label{eq:Tmid}
T_{mid}(|x-z_1|) = t^{-\frac{n-1}{2}}\int_{t^{-\frac12}}^1 
e^{it\lambda^2}\Bigl[\sum_{k+ \ell \leq \frac{n-1}{2}} c_{k,\ell} \lambda^{k+\ell-2} 
F_\ell(\lambda, |x-z_1|)
\chi^{(k)}(\lambda)\Bigr]\,d\lambda \\
+ \sum_{\ell =0}^{(n-3)/2} c_\ell \,t^{\frac{2-n-\ell}{2}} 
F_\ell(t^{-\frac12}, |x-z_1|),
\end{multline}
where $F_\ell(\lambda, |x-z_1|) = |x-z_1|^{\ell} K_+^{(\ell)}(\lambda |x-z_1|)$.
By \eqref{eq:K_uniform}, we can control 
\begin{equation} \label{eq:F_ell}
|F_\ell(\lambda, |x-z_1|)| \les |x-z_1|^2 \lambda^{2-\ell} \quad \text{for all} \quad 0\leq\ell \leq \frac{n+3}{2}
\end{equation}
As a result, each term in the last sum is bounded by $|t|^{-\frac{n}{2}}|x-z_1|^2$.

In the generic case, where one has some tolerance for weights, it suffices to integrate by
parts again to obtain the expression
\begin{multline} \label{eq:Tmid_simple}
T_{mid}(|x-z_1|) = t^{-\frac{n+1}{2}} \int_{t^{-\frac12}}^1 e^{it\lambda^2}
\Bigl[ \sum_{k+\ell \leq \frac{n+1}{2}} c_{k,\ell} 
\lambda^{k+\ell-4} F_\ell(\lambda, |x-z_1|) \chi^{(k)}(\lambda)\Bigr]\,d\lambda \\
+ \sum_{\ell =0}^{(n-1)/2} c_\ell \,t^{\frac{2-n-\ell}{2}} 
F_\ell(t^{-\frac12}, |x-z_1|).
\end{multline} 
Using \eqref{eq:F_ell} again, the boundary term is smaller than $|t|^{-\f n2}$, as
are all the integral terms corresponding to the choice $k=0$.  When $k \geq 1$,
the support of $\chi^{(k)}(\lambda)$ permits a bound of $|t|^{-\frac{n+1}{2}}|x-z_1|^2$,
which is an improvement for times $|t| >1$.

If $1VP_e = 0$, we would like to take
systematic advantage of cancellation by inserting $F_\ell(\lambda,|x-z_1|) - F_\ell(\lambda,|x|)$
every place where $F_\ell(\lambda,|x-z_1|)$ appears in~\eqref{eq:Tmid}.  In the case
where all $x_jVP_e = 0$ as well, the second-order remainder
$F_\ell(\lambda,|x-z_1|) - F_\ell(\lambda,|x|) + \nabla F_\ell(\lambda,|x|) \cdot z_1$ may be used.
To estimate these differences we note that for general radial functions
\begin{align*}
\big| F(|x-z_1|) - F(|x|)\big| &\leq \la z_1 \ra  \max_z |F'(z)| \\
\big| F(|x-z_1|) - F(|x|) + F'(|x|){\textstyle \frac{x}{|x|}}\cdot z_1\big| 
&\leq \la z_1 \ra^2 \max_z (|F''(z) + {\textstyle \frac{1}{z}} |F'(z)|)
\end{align*}
The maxima are taken over $z$ along the line segment joining $x$ to $x-z_1$, or perhaps
over $|z| \leq |x| + |z_1| \leq 2 \la x\ra \la z_1\ra$.
As applied to the functions $F_\ell(\lambda, z) = |z|^\ell K_+^{(\ell)}(\lambda z)$, this yields the
bounds
\begin{align}
\big| F_\ell(\lambda,|x-z_1|) - F_\ell(\lambda,|x|)\big| &\les \lambda^{2-\ell}\la x\ra \la z_1\ra^2  
\quad \text{ for } 0 \leq \ell \leq \frac{n+1}{2}.  \label{eq:F_ell_diff}\\
\big| F_\ell(\lambda,|x-z_1|) - F_\ell(\lambda,|x|) + \nabla F_\ell(\lambda,|x|)\cdot z_1\big| 
&\les \lambda^{2-\ell} \la z_1\ra^2 \quad \text{ for } 0 \leq \ell \leq \frac{n-1}{2}.
\label{eq:F_ell_second_diff}
\end{align}
The restrictions on $\ell$ come from the range of derivatives for which~\eqref{eq:K_uniform}
is valid.

Suppose that $1VP_e = 0$.  One can apply~\eqref{eq:F_ell_diff} to every term
in~\eqref{eq:Tmid_simple} to obtain
\begin{equation*}
|T_{mid}(|x-z_1|) - T_{mid}(|x|)| \les |t|^{-\frac{n}{2}} \la x\ra \la y\ra^2
\end{equation*}
As $VP_e$ is a bounded map from $L^1$ to $L^{1,2}$ it follows that
$\Vert \la x\ra^{-1}T_{mid}VP_e\Vert_{L^1\to L^{\infty}} \les |t|^{-\f n2}$.

In the case where $1VP_e = 0$ and $xVP_e = 0$, one can
still apply~\eqref{eq:F_ell_second_diff} and integration by parts to every
expression in~\eqref{eq:Tmid} except for the term with $k=0$ and
$\ell = \frac{n-1}{2}$.  In each case the result is consistent with the expected
bound
\begin{equation*}
\big|T_{mid}(|x-z_1|) - T_{mid}(|x|) + \nabla T_{mid}(|x|) \cdot z_1 \big|
\les |t|^{-\frac{n}{2}}\la z_1 \ra^2
\end{equation*}
The remaining task is to show that the same is true for the last integral term
\begin{equation*}
t^{\frac{1-n}{2}}\int_{t^{-\frac12}}^1 e^{it\lambda^2} \lambda^{\frac{n-5}{2}}
\big(F_{\frac{n-1}{2}}(\lambda,|x-z_1|) - F_{\frac{n-1}{2}}(\lambda, |x|)
+\nabla F_{\frac{n-1}{2}}(\lambda,|x|)\cdot z_1\big) \chi(\lambda)\,d\lambda
\end{equation*}
The problem here is that if one integrates by parts immediately, the Taylor remainder
estimate will come down to $D^2_{z_1} F_{\frac{n+1}{2}}(\lambda,|x-z_1|)$,
and~\eqref{eq:K_uniform} is false for $z>1$ if we're taking that many derivatives.
What goes wrong specifically is that the oscillatory function is really
$e^{it\lambda^2}e^{i\lambda|x-z_1|}$ with the second part coming from the
resolvent kernel.  The stationary phase point is therefore $\lambda = -|x-z_1|/2t$
instead of zero.

The dependence on $z_1$ is inconvenient, since $z_1$ is also a variable
in the next operator down the line.  To work around it, we adopt the
technique, due to Yajima, of placing the stationary phase point at $\lambda_0 = -|x|/2t$
instead.  The discrepancy between $|x|$ and $|x-z_1|$ will cost us a factor of
$\la z_1\ra$ which gets absorbed into the decay of the potential.

To set up the calculation, let 
$\tilde{F}(\lambda,x,y) = \lambda^{\frac{n-5}{2}}
e^{-i\lambda|x|}F_{\frac{n-1}{2}}(\lambda, |x-z_1|)$.
We will apply the general stationary phase bound
from Lemma~\ref{lem:statphase},
\begin{equation*} \label{eq:StatPhase}
\Big| \int_{t^{-\frac12}}^1 e^{it(\lambda-\lambda_0)^2}F(\lambda,x,z_1)\chi(\lambda)d\lambda\Big|
\les |t|^{-\frac12} \sup_\lambda |F(\lambda,x,z_1)| + |t|^{-\frac34}
\Big[\int_{t^{-\frac12}}^1 \big| {\textstyle \frac{\partial F}{\partial \lambda}}(\lambda,x,z_1)\big|^2\,d\lambda\Big]^{\frac12}.
\end{equation*} 
In particular we need to
apply Lemma~\ref{lem:statphase} to the functions
$D^2_{z_1} \tilde{F}(\lambda,x,z)$ and determine what bounds hold uniformly over $x \in \R^n$
and $|z| \leq |z_1|$.

The supremum estimate for $D^2_{z_1}\tilde{F}$ is essentially the same as~\eqref{eq:F_ell_second_diff}.
With $z_{1j}$ the $j^{th}$ component of $z_1$,
\begin{align*}
D^2_{z_{1i}z_{1j}}\tilde{F}(\lambda,x,z) &= \lambda^{\frac{n-5}{2}}e^{-i\lambda|x|}
D^2_{z_{1i}z_{1j}} F_{\frac{n-1}{2}}(\lambda,|x-z|) \\
& = \frac{(x_i-z_i)(x_j-z_j)\lambda{\frac{n-3}{2}}e^{-i\lambda|x|}}{|x-z|^4}
\Big(\lambda F_{\frac{n+3}{2}}(\lambda,|x-z|) 
- 2F_{\frac{n+1}{2}}(\lambda,|x-z|)\Big) \\
&\hskip 1in + \delta_{ij} \frac{\lambda^{\frac{n-3}{2}}}{|x-z|^2}e^{-i\lambda|x|}F_{\frac{n+1}{2}}(\lambda,|x-z|).
\end{align*}
These expressions are all of unit size by~\eqref{eq:F_ell},
therefore 
\begin{equation} \label{eq:tildeF_bestbound}
\big|\tilde{F}(\lambda,x,z_1) - \tilde{F}(\lambda,x,0)
 - \nabla_y\tilde{F}(\lambda,x,0)\cdot z_1\big| \les \la z_1 \ra^2 \text{ for all }\lambda \geq 0.
\end{equation}
Most of the estimates for $\D_\lambda D^2_{z_1}\tilde{F}(\lambda,x,z)$ follow a similar nature.
If the $\lambda$-derivative falls on a power of $\lambda$, the resulting expression will be
controlled by $\lambda^{-1}$, and $\int_{t^{-\frac12}}^1 \lambda^{-2}\,d\lambda \leq |t|^{\f12}$.

For $\frac{d}{d\lambda}[e^{-i\lambda|x|}F_\ell(\lambda,|x-z|)]$, we use the more detailed
information in~\eqref{eq:K_estimates} to conlude that
\begin{equation*}
e^{-i\lambda|x|}F_{\ell}(\lambda,|x-z|) \sim \left\{\begin{aligned}
&|x-z|^2 e^{-i\lambda|x|} \lambda^{2-\ell}, &\text{ if } \lambda |x-z| \leq 1 \\
\frac{|x-z|}{\lambda^{\ell-1}}&e^{-i\lambda|x|} 
+ \frac{e^{i\lambda(|x-z| - |x|)}}{|x-z|^{\frac{n-1}{2}-\ell}\lambda^{\frac{n-1}{2}}},
&\text{ if } \lambda |x-z| > 1
\end{aligned} \right.
\end{equation*}
In the regime where $\lambda|x-z| \leq 1$, this yields
\begin{equation*}
\frac{d}{d\lambda}[e^{-i\lambda|x|}F_\ell(\lambda,|x-z|)] \les |x-z|^2\lambda^{2-\ell}
(|z| + |x-z| + \lambda^{-1}) \les |x-z|^2\lambda^{2-\ell}(|z| + \lambda^{-1})
\end{equation*}
In the regime where $\lambda|x-z| > 1$, it gives
\begin{align*}
\frac{d}{d\lambda}[e^{-i\lambda|x|}&F_\ell(\lambda,|x-z|)] \\
&\begin{aligned}
&\les
|x-z|\lambda^{1-\ell}(|z| + |x-z| + \lambda^{-1}) + \frac{|x-z|^2\lambda^{2-\ell}}
{(\lambda |x-z|)^{\frac{n+3}{2}-\ell}}(\lambda^{-1} + (|x-z|-|x|)) \\
&\les |x-z|^{2}\lambda^{2-\ell}(|z| + \lambda^{-1})
\end{aligned}
\end{align*}
so long as $\ell \leq \frac{n+3}{2}$.  The end result is that
$|\frac{d}{d\lambda} D^2_{z_1}\tilde{F}(\lambda,x,z)| \les |z| + \lambda^{-1}$, and
consequently
\begin{equation*}
\Big| \frac{d}{d\lambda}\big[\tilde{F}(\lambda,x,z_1) - \tilde{F}(\lambda,x,0)
 - \nabla_y \tilde{F}(\lambda,x,0)\cdot z_1\big] \Big|
\les \la z_1 \ra^3 + \la z_1 \ra^2\lambda^{-1}  \text{ for all }\lambda \geq 0.
\end{equation*}
Over the interval $[t^{-1/2},1]$, the $L^2$ norm of this function is bounded
by $|t|^{\f 14}\la z_1\ra^3$ for $|t| \geq 1$.

Finally, once this estimate and ~\eqref{eq:tildeF_bestbound} are applied
in the context of Lemma~\ref{lem:statphase}, we can conclude that
\begin{equation*}
\big|T_{mid}(|x-z_1|) - T_{mid}(|x|) + \nabla T_{mid}(|x|) \cdot yz_1 \big|
\les |t|^{-\frac{n}{2}}\la z_1 \ra^3.
\end{equation*}
Under the assumptions that $1VP_e$ and $x_jVP_e$
all vanish, and with the decay assumptions on $V$,
$VP_e$ maps into the weighted space $L^{1,3}(\R^n)$,
it follows that $\Vert T_{mid}VP_e f\Vert_\infty \les |t|^{-\f n2}\Vert f \Vert_1$
for all $|t| \geq 1$.

\end{proof}

As before, we have the following immediate corollary.

\begin{corollary}

$P_eVT_{mid}$ is  a bounded operator from $L^{1,2}$ to $L^{\infty}$ whose norm is dominated
by $|t|^{-\frac{n}{2}}$.  If $P_eV1 = 0$, the same is true with estimates from $L^{1, 1}$.
If both $P_eV1 = 0$ and $P_eVx_j = 0$ for each $j \in \{1\ldots n\}$,
then the operator bound is valid between $L^1$ and $L^\infty$.

\end{corollary}

\section{Dispersive bounds for the finite Born series}\label{sec:finitebs}

In this section we consider the contribution of the finite
Born series terms, \eqref{eq:finitebs}.  We prove

\begin{prop}\label{bsprop}

	The contribution of \eqref{eq:finitebs} to
	\eqref{Stone} is bounded by $|t|^{-\f n2}$
	uniformly in $x$ and $y$.  That is,
	$$
		\sup_{x,y\in \R^n}\bigg|
		\int_0^\infty e^{it\lambda^2} \lambda \chi(\lambda)\bigg[\sum_{k=0}^{2m+1}(-1)^k\big\{
		R_0^+(VR_0^+)^k
		-R_0^-(VR_0^-)^k\big\}\bigg](\lambda^2)(x,y)\, 
		d\lambda\bigg| \les |t|^{-\f n2}.
	$$

\end{prop}  

For the first term of the Born
series, when $k=0$ in \eqref{eq:finitebs}, we define
    $$\mathcal G_1(\lambda,r)=C_1\frac{e^{i\lambda r}}{\lambda}.$$   
We then note the identity
    \begin{lemma}\label{recurrence}
            For $n\geq 3$ and odd, the following recurrence relation holds.
        \begin{align*}
            \left(\frac{1}{\lambda}\frac{d}{d\lambda}\right) \mathcal G_n       
            (\lambda, r)
            =\frac{1}{2\pi}\mathcal G_{n-2}(\lambda,r).
        \end{align*}
    
    \end{lemma}
    
    \begin{proof}
        The proof follows from the recurrence relations of the Hankel functions,
        found in \cite{AS} and the representation of the kernel given in
        \eqref{Hankel}. One can also prove this (with a fixed constant instead 
        of $2\pi$) directly using \eqref{Gn poly}.
    \end{proof}

This `dimension reduction' identity says that, up to a constant factor, the operation of $\frac{1}{\lambda}\frac{d}{d\lambda}$ takes an $n$ dimensional free resolvent to an $n-2$ dimensional
free resolvent.

\begin{lemma}\label{lem:bstn2}

	We have the bound
	$$
		\sup_{x,y\in \R^n}
		\bigg|\int_0^\infty e^{it\lambda^2} \lambda
		\chi(\lambda) [R_0^+(\lambda^2)-R_0^-(\lambda^2)]
		(x,y)\, d\lambda\bigg| \les |t|^{-\f n2}.
	$$

\end{lemma}

\begin{proof}

We note that, by Lemma~\ref{lem:R0exp} we can safely
integrate by parts $\frac{n-1}{2}$ times without
boundary terms or growth in $|x-y|$.  We consider the
case when all derivatives act  on the resolvents.
In this case, using Lemma~\ref{recurrence}, we have
to bound
\begin{align*}
	\bigg|\frac{1}{(2it)^{\frac{n-1}{2}}}\int_0^\infty
	e^{it\lambda^2} &\chi(\lambda) \lambda\bigg(
	\frac{1}{\lambda}\frac{d}{d\lambda}
	\bigg)^{\frac{n-1}{2}} [R_0^+(\lambda^2)-
	R_0^-(\lambda^2)](x,y)\, d\lambda \bigg|\\
	&\les \frac{1}{|t|^{\frac{n-1}{2}}} 
	\bigg| \int_0^\infty
	e^{it\lambda^2} \chi(\lambda)(e^{i\lambda |x-y|}
	-e^{-i\lambda |x-y|})
	\, d\lambda\bigg| \les |t|^{-\f n2}.
\end{align*}
Here the last half power of time decay follows from
Parseval and the facts that $\| \chi^\vee(\cdot \pm |x-y|)\|_1 \les 1$ uniformly in $x$ and $y$
and $\|\widehat{e^{it(\cdot)^2}}\|_\infty \les |t|^{-\f 12}$.

In the case in which one (or more) derivatives act on
the cut-off we can integrate by parts at least 
$\frac{n+1}{2}$ times.  The derivative bounds in 
Lemma~\ref{lem:R0exp} hold without any growth in
$|x-y|$ since at most $\frac{n-1}{2}$ derivatives act
on the error term and $\chi^{(k)}(\lambda)$ is supported on the
set $\lambda \approx 1$.

\end{proof}

\begin{lemma}\label{lem:BS}

	For $k\geq 1$,
	we have the bound
	\begin{align*}
		\sup_{x,y\in \R^n}\bigg|
		\int_0^\infty e^{it\lambda^2}\lambda \chi(\lambda)
		[(R_0^+(\lambda^2)V)^k R_0^+(\lambda^2)
		-(R_0^-(\lambda^2)V)^k R_0^-(\lambda^2)](x,y)\, 
		d\lambda\bigg| \les |t|^{-\f n2},
	\end{align*}
	provided $|V(x)|\les \la x\ra^{-\frac{n+3}{2}-}$.

\end{lemma}

\begin{proof}

We first note that, by Lemma~\ref{lem:R0exp}, we have
\begin{align}\label{eq:ek}
	(R_0^\pm(\lambda^2)V)^k R_0^\pm(\lambda^2)
	=K_0+\lambda^2 K_2+\dots +\lambda^{n-3} K_{n-3}
	+E_k^\pm(\lambda)
\end{align}
Here $K_j$ are real-valued, absolutely bounded operators.
One can identify these explicitly, for instance
$$
	K_0=G_0(VG_0)^{k}, \qquad
	K_2=\sum_{j=0}^k (G_0V)^jG_2(VG_0)^{k-j},
$$
though this is not important in our analysis.
Further, $E_k^{\pm}(\lambda)$ satisfy the bounds
\begin{align}
	\big|\partial_\lambda^j E_k^\pm(\lambda)\big|
	\les \lambda^{n-2-j}, \qquad j=0,1,\dots, \frac{n-1}{2}.
\end{align}
That is to say $E_k^\pm(\lambda)=\widetilde O_{\frac{n-1}{2}} (\lambda^{n-2})$.  One can also
write $E_k^\pm(\lambda)=\la x\ra \la y \ra 
\widetilde O_{\frac{n+1}{2}} (\lambda^{n-2})$ to see
that one can attain the $|t|^{-\f n2}$ decay rate as
an operator between weighted spaces.
The $\frac{n+1}{2}$st derivative requires slightly 
more care to avoid spatial weights.  

We need to bound
\begin{multline*}
	\int_0^\infty e^{it\lambda^2}\lambda \chi(\lambda)
	[(R_0^+(\lambda^2)V)^k R_0^+(\lambda^2)
	-(R_0^-(\lambda^2)V)^k R_0^-(\lambda^2)](x,y)\, 
	d\lambda\\
	=\int_0^\infty e^{it\lambda^2} \lambda \chi(\lambda)
	[E_k^+(\lambda)-E_k^-(\lambda)](x,y)\, d\lambda.
\end{multline*}
The error bounds on $E_k^{\pm}(\lambda)$ formally 
allow us to integrate by part $\frac{n-1}{2}$ times
to bound
\begin{align*}
	\frac{1}{(2it)^{\frac{n-1}{2}}} \int_0^\infty
	e^{it\lambda^2}\lambda \bigg(\frac{1}{\lambda}
	\frac{d}{d\lambda}\bigg)^{\frac{n-1}{2}} \chi(\lambda)
	[E_k^+(\lambda)-E_k^-(\lambda)](x,y)\, d\lambda.
\end{align*}
As in the free case, if at least one derivative acts on the cut-off,
the unweighted bound is clear.  
On the other hand, if all the derivatives
act on the error functions, we need only worry (about
weights) if all the derivatives act on either the
first (leading) resolvent 
$R_0^\pm(\lambda^2)(x,z_1)$ 
or the last (lagging) resolvent
$R_0^\pm(\lambda^2)(z_k,y)$.
Without loss of generality, we consider the first case
and no longer count on the `+/-' cancellation.
Then, via the `dimension reduction' identity, 
Lemma~\ref{recurrence}, we have to bound
$$
	\frac{1}{(2it)^{\frac{n-1}{2}}}
	\int_0^\infty e^{it\lambda^2} e^{\pm i\lambda|x-z_1|}
	V(R_0^\pm (\lambda^2)V)^{k-1} R_0^\pm (\lambda^2)(z_1,y)
	\, d\lambda	
$$
To see this, we note that the first resolvent contributes
$$
	R_0^\pm (\lambda^2)(x,z_1)-G_0-\lambda^2 G_2-
	\dots -\lambda^{n-3}G_{n-3}
$$
to $E_k^\pm (\lambda)$.  After $\frac{n-1}{2}$ integration
by parts, the contribution is
\begin{multline*}
	\lambda \bigg(\frac{1}{\lambda}
	\frac{d}{d\lambda}\bigg)^{\frac{n-1}{2}}
	\bigg[R_0^\pm (\lambda^2)(x,z_1)-G_0-\lambda^2 G_2-
	\dots -\lambda^{n-3}G_{n-3}\bigg]\\
	=\lambda \bigg(\frac{1}{\lambda}
	\frac{d}{d\lambda}\bigg)^{\frac{n-1}{2}}\bigg[
	R_0^\pm  (\lambda^2)(x,z_1)\bigg]
	=C_1 e^{\pm  i\lambda |x-z_1|}.
\end{multline*}
Here we use the technique of moving the stationary
point of the phase to integrate by parts
another time.  Without loss of generality we take
$t>0$, the case of $t<0$ follows with minor adjustments.
In particular, we have to extract another
$t^{-\f12}$ from
\begin{align*}
	\int_0^\infty e^{it\lambda^2\mp i\lambda|x|}& e^{\pm  i\lambda(|x-z_1|-|x|)}
	V(R_0^\pm (\lambda^2)V)^{k-1} R_0^\pm (\lambda^2)(z_1,y)
	\, d\lambda\\
	&= 	\int_0^\infty e^{it\lambda^2\mp i\lambda|x|} 
	e^{\pm i\lambda(|x-z_1|-|x|)}
	\widetilde O_1(1)	\, d\lambda
\end{align*}
Where the bounds on $V(R_0^\pm (\lambda^2)V)^{k-1} R_0^\pm (\lambda^2)$ follow from Lemma~\ref{lem:R0exp}.
We then break up the integral into two pieces, on
$0<\lambda<t^{-\f 12}$ the extra $t^{-\f12}$ decay is
easy to see.  
Using Lemma~\ref{lem:statphase} with
$$
	F(\lambda,x,z_1)=e^{-i|x|^2/4t^2}
	e^{\pm i\lambda(|x-z_1|-|x|)} \widetilde O_1(1),
$$
we can gain the extra $t^{-\f12}$ time decay on the
remaining piece $t^{-\f12}<\lambda<1$.
We can see that $\sup_\lambda |F(\lambda,x,z_1)|\les 1$
and $|\partial_\lambda F(\lambda,x,z_1)| \les \la z_1 \ra 
+\lambda^{-1}$.  So that
\begin{align*}
	t^{-\f34}\bigg(\int_{t^{-\f12}}^1 |\partial_\lambda F(\lambda,x,z_1)|^2\, d\lambda\bigg)^{\f12} &\les
	t^{-\f34}\bigg(\int_{t^{-\f12}}^1 \la z_1\ra^2 +
	\lambda^{-2}\, d\lambda	\bigg)^{\f12}
	\les t^{-\f34}\bigg(\la z_1\ra^2+t^{\f12}
	\bigg)^{\f12}\\
	&\les t^{-\f34}\la z_1 \ra +t^{-\f12}\les
	t^{-\f12} \la z_1 \ra .
\end{align*}
The growth in $z_1$ can be absorbed by decay of the
potential $V(z_1)$.

The choice of decay rate on the potential is chosen
to control the spatial integrals which arise in the iteration of resolvents.
We note that
\begin{align}\label{eq:R0bs}
	|\partial_\lambda^j R_0^\pm(\lambda^2)(x,y)| \les 
	|x-y|^{j+2-n}+\lambda^{\frac{n-3}{2}}
	|x-y|^{j+\frac{1-n}{2}}.
\end{align}
The bounds here are developed in the proof of
Lemma~\ref{lem:R0exp}.
The decay required on the potential is dictated by
the second term.  Using this as our primary bound, the terms in the integral of kernel of $\partial_{\lambda}^{\frac{n+1}{2}}(R_0^\pm(\lambda^2)V)^kR_0^\pm(\lambda^2)$ which require the most decay from
the potential are of the form
$$
	\int_{\R^{kn}}\frac{1}{|x-z_1|^{\frac{n-1}{2}-\alpha_0}}
	\prod_{j=1}^k \frac{V(z_j)}{|z_j-z_{j+1}|^{\frac{n-1}{2}-\alpha_j}}
	d\vec{z},
$$
where $\alpha_j\in \mathbb N_0$ and $\sum \alpha_j=\frac{n+1}{2}$, $z_{k+1}=y$ and
$d\vec z=dz_1\, dz_2\, \cdots \, dz_k$.
(There is of course the caveat  that
if say $\alpha_0=\frac{n+1}{2}$ the last power of 
$|x-z_1|$ is really replaced with $\la z_1\ra $, and
similarly if $\alpha_k=\frac{n+1}{2}$, the 
last $|z_k-y|$ is replaced by $\la z_k \ra$.)
Arithmetic-geometric mean inequalities allow us to consider instead the integral
\begin{align*}
	\int_{\R^{kn}}\frac{1}{|x-z_1|^{\frac{n-1}{2}}}
	\prod_{j=1}^k \frac{V(z_j)}{|z_j-z_{j+1}|^{\frac{n-1}{2}}}
	\bigg(&\la z_1\ra |x-z_1|^{\frac{n-1}{2}}\\
	&+
	\sum_{\ell=2}^{k-1} |z_\ell-z_{\ell+1}|^{\frac{n+1}{2}} +\la z_k\ra |z_k-y|^{\frac{n-1}{2}}
	\bigg)
	d\vec{z},
\end{align*}
as the quantity in parentheses dominates any product of $|z_\ell - z_{\ell+1}|$ of order $\frac{n+1}{2}$.
Choose a representative element from the summation over $\ell$, this negates a factor of
$|z_\ell - z_{\ell+1}|^{(1-n)/2}$ in the product
and replaces it with $|z_\ell - z_{\ell+1}| \les \la z_\ell\ra\la z_{\ell+1}\ra$.  Then we may consider
$$
	\int_{\R^{kn}}\biggl(\frac{1}{|x-z_1|^{\frac{n-1}{2}}}
	\prod_{j=1}^{\ell-1} \frac{V(z_j)}{|z_j-z_{j+1}|^{\frac{n-1}{2}}}\la z_\ell\ra\biggr)
	\biggl(\la z_{\ell+1}\ra \prod_{j=\ell}^k \frac{V(z_j)}{|z_j - z_{j+1}|^{\frac{n-1}{2}}}\biggr) 
	d\vec{z}.
$$
Assuming $|V(z)| \les \la z\ra^{-\beta}$, the integral over $z_\ell$ takes the form
\begin{align}\label{eq:simpiterated}
	\int_{\R^n} \frac{\la z_\ell \ra^{1-\beta}}{|z_{\ell-1}-z_\ell|^{\frac{n-1}{2}}}\, dz_\ell.
\end{align}
If $\beta>\frac{n+3}{2}$, using Lemma~\ref{EG:Lem} it is clear that
$$
	\sup_{z_{\ell-1}\in \R^n} |\eqref{eq:simpiterated}|
	\les 1.
$$
After this, the integral over $z_{\ell-1}$ can be bounded by~\eqref{eq:simpiterated}, then
the integral over $z_{\ell-2}$ and so on.  The integrals over $z_{\ell+1}, z_{\ell+2}, \ldots, z_k$ are treated
in an identical manner so the the entire integral with respect to $\vec{z}$ is bounded uniformly in $x$ and $y$.


If one uses the first term in the bound~\eqref{eq:R0bs} exclusively, there are some problems with
local singularities due to the fact that, for example
by Corollary~\ref{EGcor},
$$
	\int_{\R^n}\frac{V(z)}{|x-z|^{n-2}|z-y|^{n-2}}\,dz \les |x-y|^{4-n}
$$
may still be singular for small values of $|x-y|$.
Fortunately, the local singularities are better behaved than the expression \eqref{eq:R0bs} would indicate
due to the cancellation of the leading terms
of the expansion in \eqref{eq:ek}.  The spatial part of the
 terms that contribute to $E_k^{\pm}(\lambda)$ are of the form
$$
	G_{0+\alpha_0}(x,z_1)\prod_{j=1}^k G_{0+\alpha_j}(z_j,z_{j+1})V(z_j)
$$
with $\sum \alpha_j=n-2$ to account for the fact that $E_k^\pm$ is of order $\lambda^{n-2}$.
Collectively there is an improvement of $n-2$ powers of local
regularity.   Where $G_{n-2}$ appears in an expression it may be represent both $G_{n-2}$ and the error term $E_0^{\pm}(\lambda)$
from Lemma~\ref{lem:R0exp}, which are both bounded by $1$
with respect to the spatial variables.
 In the worst case with respect to the spatial singularities, all the $\lambda$ derivatives act on the cut-off
function $\chi$ rather than on resolvents.  Then for a fixed value of $\lambda$
we can bound their contribution by
\begin{align*}
	\int_{\R^{kn}}\frac{1}{|x-z_1|^{n-2}}
	\prod_{j=1}^k \frac{V(z_j)}{|z_j-z_{j+1}|^{n-2}}
	\bigg(& |x-z_1|^{n-2}\\
	&+
	\sum_{\ell=2}^{k-1} |z_\ell-z_{\ell+1}|^{n-2} + |z_k-y|^{n-2}
	\bigg)
	d\vec{z}.
\end{align*}
The quantity in parentheses dominates any product of $|z_\ell - z_{\ell+1}|$ of order $n-2$.
If $\beta>2$, we note that any representative term is controlled by the bound
\begin{align}\label{eq:simpiterated2}
	\sup_{z_{j-1}\in \R} \int_{\R^n} \frac{\la z_j \ra^{-\beta}}{|z_{j-1}-z_j|^{n-2}}\, dz_j \les 1
\end{align}
iterated $k$ times, again starting with $z_\ell$ and $z_{\ell+1}$.	Any other integrals with respect
to $\vec z$ can be bounded by a combination of the
two cases considered.
	
\end{proof}

We can now prove the main Theorem.

\begin{proof}[Proof of Theorem~\ref{thm:main}]

	We note that the Theorem is proven by bounding the
	oscillatory integral in the Stone formula
	\eqref{Stone},
\begin{align*}
	\bigg|
	\int_0^\infty e^{it\lambda^2} \lambda \chi(\lambda)
	[R_V^+(\lambda^2)-R_V^-(\lambda^2)](x,y))\, 
	d\lambda\bigg| \les_{x,y} |t|^{-\alpha}.
\end{align*}	
We begin by proving Part~(\ref{thmpart1}), where there is
no $x,y$ dependence.  The proof follows by 
expanding $R_V^\pm(\lambda^2)$ into the Born series
expansion, \eqref{eq:finitebs} and \eqref{eq:bstail}.
The contribution of \eqref{eq:finitebs} is bounded by
$|t|^{-\f n2}$ by Proposition~\ref{bsprop}, while
the contribution of \eqref{eq:bstail} is bounded
by $C_n|t|^{2-\f n2}P_eV1VP_e+O(|t|^{2-\f n2})$
by Lemma~\ref{lem:nocanc}.

To prove Part~(\ref{thmpart2}), one uses 
Lemma~\ref{lem:1canc} in the place of Lemma~\ref{lem:nocanc}
in the proof of Part~\ref{thmpart1}.  Finally,
Part~(\ref{thmpart3}) is proven by using
Lemma~\ref{lem:2canc}.

\end{proof}

We note that the proof of Theorem~\ref{thm:reg} is
actually simpler.  If zero is regular
$S_1=0$, so many terms drop out of the expansions.  In
addition, the expansion 
of $M^{\pm}(\lambda)^{-1}$ is of the same form
with respect to the spectral variable $\lambda$ as
$(M^{\pm}(\lambda)+S_1)^{-1}$ given in 
Lemma~\ref{M+S inverse} with different operators, see
Remark~\ref{rmk:reg}.
The dispersive bounds follow as in the analysis when
zero is not regular.

\section{Spectral characterization}\label{sec:Spec}

We prove a characterization of the spectral subspaces of
$L^2(\R^n)$ that are related to the invertibility of
certain operators in our expansions.  This characterization
is essentially Lemmas~5--7 of \cite{ES}
for three-dimensional Schr\"odinger operators modified as 
needed to fit higher spatial dimensions.  Similar
characterizations appear in Section~5 of \cite{EG} 
for two-dimensional
Schr\"odinger operators and Section~7 of 
\cite{EGG} for four
dimensions.  The proofs here are slightly simpler, as
we do not need to account for zero energy resonances.
In this section it is assumed  that $V$ is not identically zero.

In contrast the rest of the paper, in this section
$n\geq 5$ need not be odd.

\begin{lemma}\label{S characterization}

	Assume that $|V(x)|\les \la x\ra^{-2\beta}$ for some $\beta\geq 2$,
	$f\in S_1L^2(\R^n)\setminus\{0\}$ for
	$n\geq 5$ iff $f=wg$ for $g\in L^2\setminus\{0\}$ such that
	$-\Delta g+Vg=0$ in $\mathcal S'$.

\end{lemma}

\begin{proof}

	First note that
	\begin{align}\label{IG0V}
	(-\Delta+V)g=0 \Leftrightarrow (I+G_0V)g=0.
	\end{align}
	Now suppose $f\in S_1L^2\setminus\{0\}$, that is $f\neq 0$ and $f\in$ker$(U+vG_0v)$.  So that
	$$
	0=([U+vG_0v]f)(x)=U(x)f(x)+c_{0} v(x) \int_{\R^n} \frac{v(y)f(y)}{|x-y|^{n-2}}\, dy.
	$$
	Or,
	$$
	f(x)+c_{0} w(x) \int_{\R^n} \frac{v(y)f(y)}{|x-y|^{n-2}}\, dy=0.
	$$
	Let 
	$$
	g(x)=-c_{0}  \int_{\R^n} \frac{v(y)f(y)}{|x-y|^{n-2}}\, dy=-G_0vf
	$$
	So that,
	$$
		g(x)=-G_0vf(x)=-G_0v(wg)(x)=-G_0Vg(x) \qquad
		\Rightarrow \qquad  (I+G_0V)g=0.
	$$
	
	To see that $g\in L^2$, we note that
	$g(x)=-C_n I_2(vf)(x)$ with $I_2$ a Riesz potential.
	Noting, for instance, Lemma~2.3 of \cite{Jen} we have $I_2:L^{2,s}\to L^{2,-s'}$
	if $s,s'\geq 0$ and $s+s'\geq 2$.  Note that if $s=\beta\geq 2$ and $s'=0$ we have $vf\in L^{2,2}$ and then 
	$I_2 (vf)\in L^{2}$.

	On the other hand, assume $g\in L^2\setminus\{0\}$ such that
	$-\Delta g+Vg=0$ in $\mathcal S'$.  Then, denoting $f=wg$, we have $f\in L^{2,\beta}$.
	\begin{align*}
	(U+vG_0v)f(x)&=U(x)f(x)+C_nv(x) \int_{\R^n} \frac{v(y)f(y)}{|x-y|^{n-2}}\, dy\\
	&=U(x)w(x)g(x)+C_n v(x)\int_{\R^n} \frac{v(y)U(y)w(y)g(y)}{|x-y|^{n-2}}\, dy\\
	&=v(x)g(x)+v(x) G_0Vg(x)=v(I+G_0V)g=0.
	\end{align*}
	Where we used the definition of $f$ and \eqref{IG0V} in the last line.  This completes the proof.

\end{proof}

\begin{lemma}\label{lem:efnLinf}

	Assume that $|V(x)|\les \la x\ra^{-\beta}$ for some $\beta>2$,	
	if $g\in L^2(\R^n)$ for $n\geq 5$ is a solution of $(-\Delta+V)g=0$
	then $g\in L^\infty(\R^n)$.

\end{lemma}

\begin{proof}

	Recall that if $(-\Delta+V)g=0$ is equivalent to
	$(I+G_0V)g=0$, so that
	$$
		g=-G_0Vg.
	$$
	Recall that $G_0(x,y)=c_0|x-y|^{2-n}$, so that we
	can write
	\begin{align}
		g&=-G_{loc}Vg-G_{dist}Vg,
	\end{align}
	where
	$$
		G_{loc}(x,y):=\frac{c_0}{|x-y|^{n-2}}
		\chi_{\{|x-y|<\epsilon_n\}}, \qquad
		G_{dist}:=G_0-G_{loc}.
	$$	
	Here $\epsilon_n>0$ is a constant chosen so that
	$$
		\sup_{x\in \R^n}
		\int_{\R^n} |G_{loc}(x,y)V(y)|\, dy <\frac{1}{2}.
	$$
	We can take $\epsilon_n<(c_0 \|V\|_\infty^2 \omega_{n-1})^{-\f12}$ with $\omega_{n-1}$ the surface
	area of the unit ball in $\R^{n-1}$.
	Then,
	\begin{align*}
		\|G_{dist}Vg\|_\infty \leq 
		\|G_{dist}V\|_{2}\|g\|_2=C\|g\|_2
	\end{align*}
	since the $G_{dist}(x,y)V(y)$ is an $L^2$ function
	of $y$ uniformly in $x$.  So that,
	\begin{align*}
		|g(x)|&=\bigg|\int_{\R^n}-(G_{loc}(x,y)V(y)
		+G_{dist}(x,y)V(y))g(y)\, dy\bigg|\\
		&\leq \bigg(\int_{\R^n} |G_{loc}(x,y)V(y)|\,dy 
		\bigg)\|g\|_\infty+ \|G_{dist}V\|_2 \|g\|_2<\frac{1}{2}\|g\|_\infty +C \|g\|_2.
	\end{align*}
	Thus
	$$
		\|g\|_\infty < 2C \|g\|_2 <\infty.
	$$

\end{proof}

\begin{lemma}\label{D1 lemma}
 
  The kernel of $S_1vG_2vS_1$ is trivial in $S_1L^2(\R^n)$
  for $n\geq 5$.

\end{lemma}

\begin{proof}
 
  Take $f\in S_1L^2$ (that is such that $(U+vG_0v)f=0$) with $S_1vG_2vS_1f=0$.  Then using \eqref{Taylor sloppy} we have
  $G_2=\lim_{\lambda\to0} \frac{R_0^{\pm}(\lambda^2)-G_0}{\lambda^2}$.
  \begin{align*}
      0&=\la  S_1vG_2vS_1f,f\ra=\la G_2vf,vf\ra=\lim_{\lambda\to0}\bigg\la \bigg(\frac{R_0^{\pm}(\lambda^2)-G_0}{\lambda^2}\bigg)vf,vf\bigg\ra\\
      &=\lim_{\lambda\to 0} \frac{1}{\lambda^2} \int ((|\xi|^2+\lambda^2)^{-1}-|\xi|^2)^{-1} \widehat{vf}(\xi)\overline{\widehat{vf}}(\xi)d\xi\\
      &=\lim_{\lambda\to0} \int \frac{1}{|\xi|^2(|\xi|^2+\lambda^2)} |\widehat{vf}(\xi)|^2\, d\xi=\int \frac{|\widehat{vf}(\xi)|^2}{|\xi|^4}\, d\xi
      =\la G_0vf, G_0vf\ra.
  \end{align*}
  Where we used the monotone convergence theorem on the last equality.  This implies that $\widehat{vf}=0$ and $vf=0$.  Thus the kernel of
  $SvG_2vS$ is trivial.

\end{proof}

Technically the above proof only applies in odd dimensions.
It is easy to adapt the proof to even dimensions using the 
expansions given in the companion paper, \cite{GGeven}
with the operator $G_1^0$ replacing $G_2$.

\begin{corollary}
 
    For $f_1, f_2\in S_1L^2$ we have the identity
    $$
    \la G_2vf_1,vf_2\ra=\la G_0vf_1, G_0vf_2\ra.
    $$

\end{corollary}

\begin{lemma}\label{lem:eproj}
 
    The projection onto the eigenspace at zero is $G_0vS_1[S_1vG_2vS_1]^{-1}S_1vG_0$.

\end{lemma}

\begin{proof}
 
  If dim $S_1L^2=N<\infty$ (see \cite{Jen}), then let $\phi_j$, $j=1,2,\dots,N$ be an orthonormal basis for
  $S_1L^2(\R^n)$.  Then
  \begin{align*}
      0&=(U+vG_0v)\phi_j,\\
      0&=(I+wG_0v)\phi_j=\phi_j+wG_0v\phi_j.
  \end{align*}
  Write $\phi_j=w\psi_j$ for $1\leq j\leq N$ with $\psi_j$ linearly independent.  So that
  \begin{align*}
      0&=w\psi_j+wG_0vw\psi_j
  \end{align*}
  and so
  $$
  0=\psi_j+G_0V\psi_j.
  $$
  So that for any $f\in L^2$ we have
  \begin{align*}
      S_1f&=\sum_{j=1}^N \la f, \phi_j\ra \phi_j, \\
      S_1vG_0f&=\sum_{j=1}^N \la S_1vG_0f, \phi_j\ra \phi_j=\sum_{j=1}^N \la f, G_0v\phi_j\ra \phi_j
      =-\sum_{j=1}^N \la f,\psi_j\ra \phi_j
  \end{align*}
  Let $A_{ij}$ be the matrix representation of 
  $S_1vG_2vS_1$ with respect to $\{\phi_j\}_{j=1}^N$.  That is,
  \begin{align*}
      A_{ij}=\la \phi_i,S_1vG_2vS_1\phi_j\ra=\la G_0v\phi_i, G_0v \phi_j\ra
      =\la G_0V\phi_i, G_0V\phi_j\ra=\la \psi_i, \psi_j\ra.
  \end{align*}
  Denoting $P_e=G_0vS_1[S_1vG_2vS_1]^{-1}S_1vG_0$, for $f\in L^2$ we have
  \begin{align*}
      P_ef&=G_0vS_1[S_1vG_2vS_1]^{-1}S_1vG_0 f=G_0vS_1[S_1vG_2vS_1]^{-1}\Big(-\sum_{j=1}^N \la f,\psi_j\ra \phi_j\Big)\\
      &=-\sum_{j=1}^N G_0vS_1[S_1vG_2vS_1]^{-1}\phi_j \la f,\psi_j\ra=\sum_{i,j=1}^N G_0vS_1(A_{ij}^{-1})\phi_i \la f,\psi_j\ra\\
      &=-\sum_{i,j=1}^N G_0v\phi_i (A_{ij}^{-1})\la f,\psi_j\ra=\sum_{i,j=1}^N (A_{ij}^{-1})\psi_i \la f,\psi_j\ra.
  \end{align*}
  For $f=\psi_k$ we have
  \begin{align*}
      P_e\psi_k&=\sum_{i,j=1}^N (A_{ij}^{-1}) \psi_i \la \psi_k, \psi_j\ra=\sum_{i,j=1}^N (A_{ij}^{-1})(A_{jk})\psi_i=\psi_k.
  \end{align*}
  Thus, we have that the range of $P_e$ is the span of $\{\psi_j\}_{j=1}^N$ and is the identity on the range of $P_e$.  Since $P_e$ is
  self-adjoint, we are done.

\end{proof}

Defining $P_e$ to be the projection onto the zero energy
eigenspace, to match the notation of the previous 
sections we have
\begin{align}\label{Pe defn}
	P_e=G_0vD_1vG_0.
\end{align}
We also
have the following corollary.

\begin{corollary}\label{Pemapping}

	$P_e$ is bounded operator from $L^1$ to $L^\infty$.

\end{corollary}

\begin{proof}

	Take $f\in L^1$, then
	\begin{align*}
		|P_e f(x)|=\bigg|\sum_{j=1}^N \la \psi_j, f\ra 
		\psi_j\bigg| \leq \sum_{j=1}^N |\la \psi_j, f\ra|
		|\psi_j| \leq \sum_{j=1}^N \|\psi_j\|_\infty^2\| f\|_{1}<
		\infty
	\end{align*}
	by Lemma~\ref{lem:efnLinf},
 	$\psi_j\in L^\infty$.

\end{proof}

\subsection{Prospective examples of $V(x)$}
Here we construct examples of potentials $V$ for which
$H=-\Delta+V$ has a zero-energy eigenvalue whose
eigenfunction $\psi$ satisfies $\int V\psi=0$ and
$\int xV\psi=0$.  This shows that the hypotheses in
Theorem~\ref{thm:main} and Corollary~\ref{cor:ugly}
can be realized.

Any solution of the equation $(-\Delta + V)\psi = 0$ satisfies the functional relation 
$-\psi = (-\Delta)^{-1} (V\psi)$, where $(-\Delta)^{-1}$
is convolution against the Green's function $C_n|x|^{2-n}$.
If $V$ is compactly supported, then $\psi$ is a harmonic function outside the
support of $V$ and decays at the rate $|x|^{2-n}$ unless
additional cancellation takes place inside the convolution integral.  
Choose any (nonempty) finite collection of points $x_i \in \R^n$, and weights
$\mu_i \in \R$ so that the signed measure
$\sum_i \mu_i \delta_{x_i}$ has vanishing moments up to $k^{\rm th}$ order.
Then let
\begin{equation*}
F_k(x) := \big|\Delta^{-1}\big({\textstyle \sum_i} \;\mu_i\delta_{x_i}\big)\big| = 
 {\textstyle - C_n \sum_i} \; \mu_i |x-x_i|^{2-n}.
\end{equation*}

When $|x| > 2\max |x_i|$, we can expand each $|x-x_i|^{2-n}$ as a Taylor series
centered at $x$, that is  $|x-x_i|^{2-n} = P_k(x_i) +
O(|x|^{1-n-k}|x_i|^{k+1})$.  The vanishing moment assumption ensures that
$\sum_i \mu_i P_k(x_i) = 0$, which leaves  $F_k(x) =\sum_i \mu_i \,O(|x|^{1-n-k}|x_i|^{k+1}) \les |x|^{1-n-k}$.

Each $x_i$ has a neighborhood
$\Omega_i$ where $F_k(x) \sim \mu_i|x-x_i|^{2-n}$ and is
therefore nonzero.
Now let $\psi_k(x)$ be any function that agrees with
$F_k(x)$ outside of the union of $\Omega_i$ and is a nonvanishing $C^2$
continuation inside.  Then $\Delta\psi_k$ is continuous with compact
support inside $\cup_i \overline{\Omega}_i$.  Finally,
\begin{equation*}
V(x) = \frac{\Delta\psi_k(x)}{\psi_k(x)}
\end{equation*}
belongs to $C_c(\R^n)$ and the Schr\"odinger operator
$H = -\Delta + V$ has $\psi_k$ as a rapidly decaying eigenfunction
at $\lambda = 0$.  Specifically, $|\psi_k(x)| \les |x|^{1-n-k}$ for large $|x|$.

The conditions of Theorem~\ref{thm:main}, part~\ref{thmpart3}
are satisfied for any potential constructed in this manner with $k \geq 1$,
provided all other eigenfunctions (if any exist) also decay at the 
rate $|x|^{-n}$ or faster.  The following argument adapted from~\cite{goldE}
suggests that typically the eigenspace is in fact one-dimensional.

Starting from a fixed choice of $\psi_k$ it is possible to construct a larger
family of complex potentials of the form 
$\psi_z(x)  = \psi_k(x)e^{z\eta(x)}$, where $\eta \in C^\infty_c(\cup_i\Omega_i)$
and $z$ varies over $\mathbb C$.  The resulting
potentials $V_z(x)$ have complex-analytic dependence on $z$.  The analytic
Fredholm theorem applied to $(I + (-\Delta)^{-1}V_z)$ indicates that the dimension
of the nullspace for $-\Delta + V_z$ should be constant for generic $z$ with only 
a discrete set of exceptions where it is larger.

If the nullspace of $H_0 := -\Delta + V_0$ is one-dimensional, we are done.
Supose the nullspace is two-dimensional with a second eigenfunction $\phi$.
As constructed, the formula for $V_z$ is
\begin{equation*}
V_z = \frac{\Delta(e^{z\eta}\psi_k)}{e^{z\eta}\psi_k} = V_0 - \frac{H_0(\eta \psi_k)}{\psi_k}z
+ \frac{|\nabla \eta|^2}{\psi_k}z^2
\end{equation*}
In particular, $\frac{d}{dz}V_z\big|_{z=0} = -\frac{H_0(\eta\psi_k)}{\psi_k}$.

Then for small values of $z$,
\begin{align*}
\la(-\Delta + V_z)\phi, \phi\ra &= \la(-\Delta + V_0) \phi,\phi\ra 
-\Big\la \Big(\frac{H_0(\eta \psi_k)}{\psi_k}\Big)\phi, \phi\Big\ra z + O(z^2) \\
&= -\Big\la H_0(\eta \psi_k) , \frac{\phi^2}{\psi_k}\Big\ra z + O(z^2).
\end{align*}

If the leading-order term is nonzero, then the repeated eigenvalue when $z=0$ will
split for all other nearby values of $z$.  However, since
$H_0$ is self-adjoint, we can rewrite the inner product as
\begin{equation*}
\Big\la H_0(\eta \psi_k), \frac{\phi^2}{\psi_k}\Big\ra = \Big \la \eta,
\psi_k H_0\Big(\frac{\phi^2}{\psi_k}\Big)\Big\ra.
\end{equation*}
Since $\psi_k$ and $\phi$ are linearly independent, $\frac{\phi^2}{\psi_k}$ is linearly independent
from both of them, thus the function on the right side of the inner product is nonzero.
Any choice of $\eta$ which is not orthogonal to this suffices for constructing potentials
$V_z$ that have a one-dimensional eigenspace at zero (with eigenvector $e^{z\eta}\psi_k$).

Similar eigenvalue-splitting arguments hold if the nullspace of $H_0$ is larger, provided a
sufficient collection of functions $H_0(\frac{\phi_i \phi_j}{\psi_k})$ are nonzero and
linearly independent.

\section{Integral Estimates}\label{sec:ests}

The proof of Lemma~\ref{lem:iterated} hinges on the following estimate.
\begin{lemma}\label{EG:Lem}

	Fix $u_1,u_2\in\R^n$ and let $0\leq k,\ell<n$, 
	$\beta>0$, $k+\ell+\beta\geq n$, $k+\ell\neq n$.
	We have
	$$
		\int_{\R^n} \frac{\la z\ra^{-\beta-}}
		{|z-u_1|^k|z-u_2|^\ell}\, dz
		\les \left\{\begin{array}{ll}
		(\frac{1}{|u_1-u_2|})^{\max(0,k+\ell-n)}
		& |u_1-u_2|\leq 1\\
		\big(\frac{1}{|u_1-u_2|}\big)^{\min(k,\ell,
		k+\ell+\beta-n)} & |u_1-u_2|>1
		\end{array}
		\right.
	$$

\end{lemma}

More precisely, noting that $\max(0,k+\ell-n)\leq \min(k,\ell,k+\ell+\beta-n)$,
we use the simple corollary,

\begin{corollary}\label{EGcor}

	Fix $u_1,u_2\in\R^n$ and let $0\leq k,\ell<n$, 
	$\beta>0$, $k+\ell+\beta\geq n$, $k+\ell\neq n$.
	We have
	$$
		\int_{\R^n} \frac{\la z\ra^{-\beta-}}
		{|z-u_1|^k|z-u_2|^\ell}\, dz
		\les 
		\bigg(\frac{1}{|u_1-u_2|}\bigg)^{\max(0,k+\ell-n)},
	$$
	and
	$$
		\int_{\R^n} \frac{\la z\ra^{-\beta-}}
		{|z-u_1|^k|z-u_2|^\ell}\, dz
		\les 
		\bigg(\frac{1}{|u_1-u_2|}\bigg)^{\min(k,\ell,
		k+\ell+\beta-n)}.
	$$

\end{corollary}

\begin{proof}[Proof of Lemma~\ref{lem:iterated}]

	We note the bound
	$$
		|R_0^\pm(\lambda^2)(x,y)|\les \frac{1}{|x-y|^{n-2}}+
		\frac{\lambda^{\frac{n-3}{2}}}
		{|x-y|^{\frac{n-1}{2}}},
	$$
	which follows from \eqref{Gn poly}.  To control
	local singularities, it is only necessary to 
	bound integrals of the form
	\begin{align*}
		\iint_{\R^{kn}}\frac{1}{|x-z_1|^{n-2}}
		\prod_{j=1}^n\frac{\la z_j \ra^{-\beta-}}
		{|z_j-z_{j+1}|^{n-2}}\, 
		d\vec z,
	\end{align*}
	where $z_{n+1}=y$ and $d\vec{z}=dz_1\, dz_2\,\dots\,
	dz_{n}$.  For this integral we iterate the first
	bound in Corollary~\ref{EGcor}.
	Since $\ell=n-2$ at each step
	we have that $k+\ell-n=k-2$, that is we improve the
	local singularity by two powers after each integral.
	So that, after $m$ integrals of this form the 
	singularity is $|x-z_{m}|^{n-2(m+1)}$, so that we
	need $n-2(m+1)\leq \frac{n-1}{2}$ to ensure the
	iterated integral is locally $L^2$.  We take $\kappa=
	\frac{n-1}{4}$ or $\frac{n-3}{4}$ whichever is an
	integer.  At the final step if $n \equiv 1 \mod 4$ we have,
	$$
		\int_{\R^n} \frac{\la z_\kappa \ra^{-\beta-}}
		{|x-z_\kappa|^{\frac{n+1}{2}}|z_\kappa-y|^{n-2}}
		\, dz_{\kappa} \les \left\{\begin{array}{ll} |x-y|^{\frac{3-n}{2}} & |x-y|\leq 1\\
		|x-y|^{-\frac{n+1}{2}} & |x-y|\geq 1		
		\end{array}
		\right. 
	$$
	which is in $L^2_y$
	uniformly in $x$.  On the other hand, if
	$n\equiv 3 \mod 4$ we have
	$$
		\int_{\R^n} \frac{\la z_\kappa \ra^{-\beta-}}
		{|x-z_\kappa|^{\frac{n+3}{2}}|z_\kappa-y|^{n-2}}
		\, dz_{\kappa} \les  \left\{\begin{array}{ll} |x-y|^{\frac{1-n}{2}} & |x-y|\leq 1\\
		|x-y|^{-\frac{n+3}{2}} & |x-y|\geq 1		
		\end{array} \right. 
	$$
	which is also in $L^2_y$
	uniformly in $x$.

	To determine the weight needed, we need only control
	integrals of the form
	\begin{align}\label{wtdL2 wts}
		\iint_{\R^{kn}}\frac{1}{|x-z_1|^{\frac{n-1}{2}}}
		\prod_{j=1}^n\frac{\la z_j \ra^{-\beta-}}
		{|z_j-z_{j+1}|^{\frac{n-1}{2}}}\, 
		d\vec z.
	\end{align}
	Here we use the second bound in
	Corollary~\ref{EGcor}
	for when $k,\ell=\frac{n-1}{2}$ we have
	$$
		\int_{\R^n} \frac{\la z\ra^{-\beta-}}
		{|z-u_1|^k|z-u_2|^\ell}\, dz
		\les 
		\big(\frac{1}{|u_1-u_2|}\big)^{\min(\frac{n-1}{2},
		\beta-1)} 
	$$
	If $\beta>\frac{n+1}{2}$, then $\min(\frac{n-1}{2},
			\beta-1)=\frac{n-1}{2}$ and 
	we can iterate this bound to
	see that
	\begin{align*}
		|\eqref{wtdL2 wts}| &\les 
		\frac{1}{|x-y|^{\frac{n-1}{2}}} \in 
		L^{2,-\f12-}_y(\R^n)
	\end{align*}
	uniformly in $x$.  Any other terms that appear in the
	product of free resolvents can be seen to be bounded
	by a sum of integrals of the form we considered here.

\end{proof}

\begin{lemma}\label{lem:IBP}

	If $k\in \mathbb N_0$, we have the bound
	$$
		\bigg|\int_0^\infty e^{it\lambda^2} \chi(\lambda)
		\lambda^k \, d\lambda\bigg| \les |t|^{-\frac{k+1}{2}}.
	$$

\end{lemma}

\begin{proof}

	We employ the identity $\frac{d}{d\lambda}e^{it\lambda^2}=2it\lambda e^{it\lambda^2}$ to integrate by parts against the
	imaginary Gaussian.  Formally, upon integrating
	by parts $m$ times we have
	\begin{align*}
		\bigg|\int_0^\infty e^{it\lambda^2} \chi(\lambda)
		\lambda^k \, d\lambda\bigg| &\les
		\sum_{j=0}^{m-1} \frac{\chi(\lambda)\lambda^{k-1-2j}}{|t|^{j+1}}
		\bigg|_{\lambda=0}^\infty
		+\frac{1}{|t|^{m}}\bigg|
		\int_0^\infty e^{it\lambda^2}
		\chi(\lambda)\lambda^{k-2m}\, d\lambda\bigg| +
		O(|t|^{-\ell}).
	\end{align*}
	The first term collects the boundary terms in the integration
	by parts, while the final $O(|t|^{-\ell})$ term 
	for arbitrary $\ell\geq 0$ is
	obtained from any term where the derivative acts on
	the cut-off function $\chi(\lambda)$.  This follows
	since $\chi^\prime(\lambda)$ is supported on 
	the annulus $\lambda \approx 1$ and one can integrate
	by parts arbitrarily many times with no boundary
	terms or convergence issues.

	If $k$ is odd, we select $m=\frac{k-1}{2}$. The boundary terms are all 
	zero since $k-1-2j\geq 2$.
	Meanwhile, one additional integration by parts yields
	$$
		\frac{1}{|t|^{\frac{k-1}{2}}}\bigg|\int_0^\infty
		e^{it\lambda^2} \lambda \chi(\lambda) 
		d\lambda\bigg| \les \frac{1}{|t|^{\frac{k+1}{2}}}
		+\frac{1}{|t|^{\frac{k+1}{2}}}
		\int_0^\infty e^{it\lambda^2} \chi^{\prime}
		(\lambda)\, d\lambda \les |t|^{-\frac{k+1}{2}}.
	$$
	
	On the other hand, if $k$ is even, we select
	$m=\frac{k}{2}$ and note that all the boundary
	terms are zero since $k-1-2j\geq 1$.
	The integral bound
	$$
		\frac{1}{|t|^{\frac{k}{2}}}\bigg|\int_0^\infty
		e^{it\lambda^2} \chi(\lambda) 
		d\lambda\bigg| \les \frac{1}{|t|^{\frac{k+1}{2}}}.
	$$
	follows from Parseval's Identity since
	$$
		\| \widehat{e^{it(\cdot)^2}}\|_\infty \les 
		|t|^{-\f12}, \qquad  \textrm{ and } \qquad
		\|\chi^{\vee}(\cdot)\|_1\les 1.
	$$

\end{proof}

The above calculations do not immediately apply to
bounding integrals of the form
$$
	\int_0^\infty e^{it\lambda^2} f(\lambda)\, d\lambda
$$
when $f$ and its derivatives are bounded by powers of
$\lambda$.  Accordingly, we have the following 
oscillatory integral bound.

\begin{lemma}\label{lem:fauxIBP}

	For a fixed $\alpha>-1$,
	let $f(\lambda)=\widetilde O_{k+1}(\lambda^\alpha)$ be supported on the 
	interval $[0,\lambda_1]$ for some $0<\lambda_1\les 1$.
	Then, if $k$ satisfies $-1<\alpha-2k< 1$ we have
	\begin{align*}
		\bigg|\int_0^\infty e^{it\lambda^2} 
		f(\lambda)\, d\lambda\bigg|&\les |t|^{-\frac{\alpha+1}{2}}.
	\end{align*}
	
\end{lemma}

\begin{proof}

	As in the proof of
	Lemma~\ref{lem:IBP}
	we can integrate by parts $k$ times without having
	boundary terms  since $\alpha-2k+2>1$.  
	At this point, we need only bound
	\begin{align}\label{ibplemeqn}
		\frac{1}{|t|^k}	\int_0^\infty e^{it\lambda^2}
		g(\lambda)\, d\lambda.
	\end{align}
	Here $g(\lambda)=\widetilde O_1(\lambda^{\alpha-2k})$
	is again supported on $[0,\lambda_1]$.  By the
	definition of the integer $k$ we have 
	$-1<\alpha-2k< 1$ so that further integration by parts is not possible.
	Without loss of generality we take $t>0$ and
	we break the integral into two parts,
	\begin{align*}
		|\eqref{ibplemeqn}| &\les \frac{1}{|t|^k}
		\int_0^{t^{-\f12}} |g(\lambda)|\,d\lambda+
		\frac{1}{|t|^k}\bigg|
		\int_{t^{-\f12}}^\infty e^{it\lambda^2} g(\lambda)
		\, d\lambda\bigg|.
	\end{align*}
	The bound for the first integral follows from integration,
	\begin{align*}
		\int_0^{t^{-\f12}} |g(\lambda)|\,d\lambda
		\les \int_0^{t^{-\f12}} \lambda^{\alpha-2k}\, 
		d\lambda \les |t|^{-\frac{\alpha}{2}+k-\f12}.
	\end{align*}
	Since $-1<\alpha-2k$ this is integrable at zero.
	For the second integral, we integrate by parts again.
	As $\alpha-2k<1$ there is no a boundary term at
	infinity, but we do have one at $t^{-\f12}$.  Thus
	we see
	\begin{align*}
		\int_{t^{-\f12}}^\infty e^{it\lambda^2} g(\lambda)
		\, d\lambda&\les\frac{g(\lambda)}{\lambda t}
		\bigg|_{\lambda=t^{-\f12}} +\frac{1}{t}
		\int_{t^{-\f12}}^\infty \lambda^{\alpha-2k-2}
		\, d\lambda \les |t|^{-\frac{\alpha}{2}+k-\f12}.
	\end{align*}
	The final integrand is integrable as $\alpha-2k-2<-1$.

\end{proof}

We need the following lemma which is a modification of
stationary phase.  

\begin{lemma}\label{lem:statphase}

We have the bound
\begin{align*} 
\Big| \int_{t^{-\frac12}}^1 e^{it(\lambda-\lambda_0)^2}F(\lambda,x,y)\chi(\lambda)d\lambda\Big|
\les t^{-\frac12} \sup_\lambda |F(\lambda,x,y)| + t^{-\frac34}
\Big[\int_{t^{-\frac12}}^1 \big| {\textstyle \frac{\partial F}{\partial \lambda}}(\lambda,x,y)\big|^2\,d\lambda\Big]^{\frac12}.
\end{align*}

\end{lemma}


\begin{proof}

	Assume that $t>0$, the
	proof for $t<0$ proceeds identically.
	We note that in the case when $\lambda_0\ll t^{-\f12}$
	or $\lambda_0\gg 1$ we can integrate by parts
	safely against $e^{it(\lambda-\lambda_0)^2}$ since
	$|\lambda-\lambda_0|\gtrsim \lambda$.  Then, we have
	\begin{multline*}
		\Big| \int_{t^{-\frac12}}^1 
		e^{it(\lambda-\lambda_0)^2}F(\lambda,x,y)
		\chi(\lambda)d\lambda\Big|\\
		\les \frac{F(\lambda,x,y)\chi(\lambda)}{\lambda t}
		\bigg|^1_{t^{-\f12}}+\frac{1}{t}\int_{t^{-\f12}}^1
		\bigg|\frac{F(\lambda,x,y)}{\lambda^2}\bigg|
		+\bigg|\frac{\partial_\lambda F(\lambda,x,y)}
		{\lambda} \bigg|\, d\lambda
	\end{multline*}
	The first two terms can be seen to be bounded by
	$t^{-\f12}\sup_\lambda |F(\lambda,x,y)|$.  The 
	second term is bounded by applying Cauchy-Schwartz.
	
	On the other hand, if $t^{-\f12}\les \lambda_0\les 1$
	we instead consider
	\begin{align*}
		e^{it\lambda_0^2}\int_{t^{-\f12}}^1 
		e^{it\lambda^2} e^{-2it\lambda_0\lambda} 
		F(\lambda,x,y)\, d\lambda
	\end{align*}
	Here we can integrate by parts once, and ignore the
	constant $e^{it\lambda_0^2}$ to bound
	\begin{align*}
		t^{-\f12}\sup_{\lambda}|F(\lambda,x,y)|
		+\frac{1}{t}\int_{t^{-\f12}}^1 |F(\lambda,x,y)|
		(\lambda_0 \lambda^{-1}+\lambda^{-2})
		+\frac{|\partial_\lambda F(\lambda,x,y)|}{\lambda}
		\, d\lambda
	\end{align*}
	Since $\lambda_0\les 1$, we see that
	\begin{align*}
		\frac{1}{t}\int_{t^{-\f12}}^1 &|F(\lambda,x,y)|
		(\lambda_0 \lambda^{-1}+\lambda^{-2})\, d\lambda
		\les \frac{\sup_\lambda |F(\lambda,x,y)|}{t}
		\int_{t^{-\f12}}^1 \bigg(\frac{\lambda_0}{\lambda}
		+\frac{1}{\lambda^2}\bigg)\, d\lambda\\
		&\les \frac{\sup_\lambda |F(\lambda,x,y)|}{t}
		\int_{t^{-\f12}}^1 \bigg(t^{\f12}
		+\frac{1}{\lambda^2}\bigg) \, d\lambda	
		\les \frac{\sup_\lambda 
		|F(\lambda,x,y)|}{t^{\f12}}	
	\end{align*}
	The final term is bounded by Cauchy-Schwartz,
	\begin{align*}
		\frac{1}{t}\int_{t^{-\f12}}^1 
		\frac{|\partial_\lambda F(\lambda,x,y)|}{\lambda}
		\, d\lambda &\les \frac{1}{t} 
		\bigg(\int_{t^{-\f12}}^{\infty} \lambda^{-2}\, 
		d\lambda \bigg)^{\f12}\bigg(\int_{t^{-\f12}}^1 
		|\partial_\lambda F(\lambda,x,y)|	\, d\lambda
		\bigg)^{\f12}\\
		&\les |t|^{-\f34}\bigg(\int_{t^{-\f12}}^1 
		|\partial_\lambda F(\lambda,x,y)|\, d\lambda
		\bigg)^{\f12}
	\end{align*}

\end{proof}

Finally, we remark that the oscillatory bounds in
Lemmas~\ref{lem:IBP} and~\ref{lem:statphase} hold if the
cut-off function $\chi(\lambda)$ is supported on
$[0,\lambda_1]$ for any finite $\lambda_1$.

\end{document}